\documentclass[oneside,english]{amsart}
\usepackage{lmodern}
\usepackage[T1]{fontenc}
\usepackage[latin9]{inputenc}
\usepackage{geometry}
\usepackage{mathrsfs}
\usepackage{amsthm}
\usepackage{amsbsy}
\usepackage{amssymb}
\usepackage{stackrel}
\usepackage{esint}

\makeatletter
\numberwithin{equation}{section}
\numberwithin{figure}{section}
\theoremstyle{plain}
\newtheorem{thm}{\protect\theoremname}
  \theoremstyle{remark}
  \newtheorem*{rem*}{\protect\remarkname}
  \theoremstyle{remark}
  
  \theoremstyle{plain}
  \newtheorem{cor}[thm]{\protect\corollaryname}
  \theoremstyle{plain}
  \newtheorem{prop}[thm]{\protect\propositionname}
  \theoremstyle{plain}
  \newtheorem{lem}[thm]{\protect\lemmaname}

\makeatother

\usepackage{babel}
  \providecommand{\corollaryname}{Corollary}
  \providecommand{\lemmaname}{Lemma}
  \providecommand{\propositionname}{Proposition}
  \providecommand{\remarkname}{Remark}
\providecommand{\theoremname}{Theorem}

\begin{document}

\title{The extremal process of critical points of the pure $p$-spin spherical
spin glass model}

\author{Eliran Subag and Ofer Zeitouni}
\thanks{E. S. acknowledges the support of the Adams Fellowship Program
	of the Israel Academy of Sciences and Humanities. 
  The work of both authors was partially supported by a US-Israel BSF grant and
by a grant from the Israel Science Foundation.}
\begin{abstract}
Recently, sharp results concerning the critical points of the Hamiltonian
of the $p$-spin spherical spin glass model have been obtained by
means of moments computations. In particular, these moments computations
allow for the evaluation of the leading term of the ground-state,
i.e., of the global minimum. In this paper, we study the extremal
point process of critical points - that is, the point process associated
to all critical values in the vicinity of the ground-state. We show
that the latter converges in distribution to a Poisson point process
of exponential intensity. In particular, we identify the correct centering
of the ground-state and prove the convergence in distribution of the
centered minimum to a (minus) Gumbel variable. These results are identical
to what one obtains for a sequence of i.i.d variables, correctly normalized;
namely, we show that the model is in the universality class of REM.
\end{abstract}

\maketitle

\section{Introduction}

The Hamiltonian of the spherical pure\emph{ }$p$-spin spin glass
model \cite{pSPSG} is given by 
\begin{equation}
H_{N}\left(\boldsymbol{\sigma}\right)=\frac{1}{N^{\left(p-1\right)/2}}\sum_{i_{1},...,i_{p}=1}^{N}J_{i_{1},...,i_{p}}\sigma_{i_{1}}\cdots\sigma_{i_{p}},\quad\boldsymbol{\sigma}\in\mathbb{S}^{N-1}\left(\sqrt{N}\right),\label{eq:Hamiltonian}
\end{equation}
where $\boldsymbol{\sigma}=\left(\sigma_{1},...,\sigma_{N}\right)$
, $\mathbb{S}^{N-1}\left(\sqrt{N}\right)\triangleq\left\{ \boldsymbol{\sigma}\in\mathbb{R}^{N}:\,\left\Vert \boldsymbol{\sigma}\right\Vert _{2}=\sqrt{N}\right\} $,
and $J_{i_{1},...,i_{p}}$ are i.i.d standard normal variables. Everywhere in the paper we shall assume that $p\geq3$. Recently,
sharp results concerning the critical points of $H_{N}$ have been
obtained by means of moments computations \cite{A-BA-C,2nd}. The
first contribution is the seminal work of Auffinger, Ben Arous and
{\v{C}}ern{\'y} \cite{A-BA-C} (see also \cite{ABA2} for the mixed
case). With $B\subset\mathbb{R}$ and ${\rm Crt}_{N}\left(B\right)$
denoting the number of critical points with critical values in $NB=\left\{ Nx:x\in B\right\} $,
they showed that, for any $p\geq3$,
\begin{equation}
\lim_{N\to\infty}\frac{1}{N}\log\left(\mathbb{E}\left\{ \mbox{Crt}_{N}\left(\left(-\infty,u\right)\right)\right\} \right)=\Theta_{p}\left(u\right),\label{eq:1st_mom}
\end{equation}
where $\Theta_{p}\left(u\right)$ is given in (\ref{eq:Theta_p})
(similar asymptotics were computed in \cite{A-BA-C} for the number
of critical points of a given index, but we shall not discuss those
in the current work). Of course, by itself, the first moment gives
limited information about the corresponding probability law. The goal
of \cite{2nd} was to address the question of concentration of the
random variable in (\ref{eq:1st_mom}) around its mean, by a second
moment computation. Set $E_{\infty}=2\sqrt{\left(p-1\right)/p}$ and
let $E_{0}>E_{\infty}$ be the unique number satisfying $\Theta_{p}\left(-E_{0}\right)=0$.
The main result of \cite{2nd} is that for $u\in\left(-E_{0},-E_{\infty}\right)$,
the ratio of the second to first moment squared of $\mbox{Crt}_{N}\left(\left(-\infty,u\right)\right)$
converges to $1$, as $N\to\infty$. Consequently, $\mbox{Crt}_{N}\left(\left(-\infty,u\right)\right)$
divided by its mean converges in $L^2$ to $1$. 

By appealing to Markov's inequality, (\ref{eq:1st_mom}) provides
a lower bound on the minimum of $H_{N}$. Using the Parisi formula
\cite{Parisi,pSPSG}, proved in \cite{Talag,Chen}, the authors of
\cite{A-BA-C} were able to derive a matching upper bound and show
that the so-called ground-state (i.e., global minimum of $H_{N}$)
satisfies
\[
\lim_{N\to\infty}\min_{\boldsymbol{\sigma}\in\mathbb{S}^{N-1}\left(\sqrt{N}\right)}H_{N}\left(\boldsymbol{\sigma}\right)/N=-E_{0},\,\,\,\,\mbox{a.s.}
\]
Alternatively, the matching upper bound can be derived from the quoted $L^2$ convergence.

The goal of the current work is to investigate the behavior of the
ground-state and the collection of critical values in its vicinity
on a finer scale than the above. Namely, we study the \emph{extremal
point process of critical points} which we define by
\begin{equation}
\xi_{N}\triangleq\left(1+\iota_{p}\right)^{-1}\sum_{\boldsymbol{\sigma}\in\mathscr{C}_{N}}\delta_{H_{N}\left(\boldsymbol{\sigma}\right)-m_{N}},\label{eq:xi_N}
\end{equation}
with 
\begin{equation}
m_{N}=-E_{0}N+\frac{\frac{1}{2}\log N}{c_{p}}-K_{0}\label{eq:m_N}
\end{equation}
where $\iota_{p}=\left(1+\left(-1\right)^{p}\right)/2$, $K_{0}$
is given in (\ref{eq:C0}), 
\begin{equation}
c_{p}\triangleq\Theta_{p}^{\prime}\left(-E_{0}\right) \label{eq:0406-1}
\end{equation}
with $\Theta_{p}^{\prime}$ denoting the derivative of $\Theta_{p}$,
and $\mathscr{C}_{N}$ denotes the set of critical points of $H_{N}$.
That is,
\[
\mathscr{C}_{N}\triangleq\left\{ \boldsymbol{\sigma}\in\mathbb{S}^{N-1}\left(\sqrt{N}\right):\nabla H_{N}\left(\boldsymbol{\sigma}\right)=0\right\} ,
\]
where $\nabla H_{N}$ is the gradient relative to the standard differential
structure on the sphere. The reason for the normalizing term preceding
the sum in (\ref{eq:xi_N}) is parity: since for any $\boldsymbol{\sigma}$
on the sphere, $H_{N}\left(\boldsymbol{\sigma}\right)=\left(-1\right)^{p}H_{N}\left(-\boldsymbol{\sigma}\right)$,
 for even $p$ the multiplicity of any critical value is even. 

Endow the space of point processes with the vague topology and denote
by $PPP\left(\mu\right)$ the distribution of a Poisson point process
(PPP) with intensity measure $\mu$. Our main result is the following. 
\begin{thm}
\label{thm:ext-proc}For any $p\geq3$ the extremal process of critical
points converges in distribution to a Poisson point process of exponential
intensity. Namely, with $c_p$ as in \eqref{eq:0406-1},
\begin{equation}
\xi_{N}\stackrel[{\scriptstyle N\to\infty}]{d}{\to}\xi_{\infty}\sim PPP\left(e^{c_{p}x}dx\right).\label{eq:xi_infty}
\end{equation}
\end{thm}
\begin{rem*}
It follows from Lemma \ref{lem:g1B1} below that for any interval $J$,
with probability approaching $1$ as $N\to\infty$, all the critical
points of $H_{N}$ in $J+m_{N}=\left\{ x+m_{N}:\,x\in J\right\} $
are  local minima. In other words, nothing changes in Theorem \ref{thm:ext-proc}
if $\xi_{N}$ is defined using only the minimum points instead of
all critical points.
\end{rem*}
As a corollary to Theorem \ref{thm:ext-proc}, we obtain explicitly
the limiting law of the ground-state.
\begin{cor}
\label{cor:GS}For any $p\geq3$ the centered ground-state converges
in distribution, as $N\to\infty$, to the negative of a Gumbel variable.
Namely, with $c_p$ as in \eqref{eq:0406-1},
\[
\lim_{N\to\infty}\mathbb{P}\left\{ \min_{\boldsymbol{\sigma}\in\mathbb{S}^{N-1}\left(\sqrt{N}\right)}H_{N}\left(\boldsymbol{\sigma}\right)-m_{N}\geq x\right\} =\exp\left\{ -\frac{1}{c_{p}}e^{c_{p} x}\right\} .
\]

\end{cor}

The Gibbs measure is the probability measure on $\mathbb{S}^{N-1}(\sqrt{N})$ with density proportional to $e^{-\beta H_N(\boldsymbol{\sigma})}$ w.r.t to standard volume measure on the sphere. At zero temperature, i.e., $\beta=\infty$, one thinks of the measure as the atomic measure concentrated at the global minimum point (or two, if $p$ is even) and Corollary \ref{cor:GS} gives the corresponding ground-state energy. For large but finite $\beta$, the Gibbs measure should be governed by low values of $H_N(\boldsymbol{\sigma})$, but it is a-priori unclear to what extent the critical values in Theorem \ref{thm:ext-proc} (i.e., values within distance $o(N)$ from $m_N$) are relevant to this. In the recent \cite{geometryGibbs} it was shown that, if $\beta$ is large enough, the Gibbs measure asymptotically concentrates on thin spherical bands, of $\beta$-dependent radius, centered at exactly those critical points. As shown there, this implies the absence of temperature chaos and yields a second order logarithmic term for the free energy as $N\to\infty$.

\subsection*{Extremal processes in related models}

To put things into context, we now discuss several models of importance,
intimately related to the spherical $p$-spin model. First are the
Random Energy Model (REM), in which energy levels are assumed to be
independent, and generalized REM (GREM), where correlations are introduced
through a tree structure of finite depth. These mathematically tractable
models were introduced by Derrida in the 80s \cite{REM,GREM} in order
to investigate the phenomenon of replica symmetry breaking exhibited
by the Sherrington-Kirkpatrick (SK) model. They have been extensively
studied since then and clear connections to spin glass theory have
been established, elucidating important concepts in the Parisi theory
(see, e.g., the review papers \cite{BovKurInCollection} by Bovier
and Kurkova and \cite{BolthausenIncollection} by Bolthausen). Wishing
to extend the tree structure mentioned above to account for infinite
number of levels one is naturally led to the Branching Brownian Motion
(BBM) and Branching Random Walk (BRW), which are of interest
on their own. Good sources about the above mentioned models, motivated
by connections to spin glass theory and very relevant to the study
of extremal processes, are the lecture notes of Bovier \cite{BovierNotes}
and Kistler \cite{KistlerNotes}. Another related model which possesses
an implicit hierarchical structure similar to BRW is the 2-Dimensional
Discrete Gaussian Free Field (DGFF), see \cite{BoltDeuschGiac}. 

The convergence of the extremal process of the REM model to a PPP
of exponential intensity is a classical result of extreme value theory
\cite{pickands,LeadLindRootz}. Already for the relatively simple
GREM the classical theory is not enough. In the two papers \cite{BovKurk1,BovKurk2}
Bovier and Kurkova studied the extreme values of the GREM model and
a generalization of it, the CREM model. For the GREM model they describe
the extremal process in terms of a cascade of PPPs. Their representation
implies, in particular, that the process is in general a randomly
shifted PPP (SPPP). Convergence in the case of BBM was established
independently in two important papers by Arguin, Bovier and Kistler
\cite{ABKextremal} and A{\"{i}}d{\'{e}}kon, Berestycki, Brunet and
Shi \cite{ABBS}, using somewhat different approaches. The limiting
process was shown to be a randomly shifted, decorated PPP (SDPPP)\footnote{A decorated Poisson point process is a process whose law is obtained
from a Poisson point process by replacing each of the atoms by an
independent copy of some point process (called the ``decoration'').
An SDPPP is simply a decorated Poisson point process, shifted by an
independent random variable.} of exponential intensity. See also Madaule \cite{MadauleBRWext},
where a similar description in the case of BRW is given. Finally, for the
2D DGFF, Biskup and Louidor \cite{Louidor} show in an important paper,
which we shall come back to later, that the extremal process corresponding
to local maxima converges to an SPPP. Their methods are expected to
also yield the convergence of the full process to an SDPPP. For the super critical case, i.e. dimension 3 and above, Chiarini, Cipriani and Hazra \cite{ECP4332} proved convergence of the extremal process of the DGFF to a PPP of exponential intensity.

The extreme
values of spin glass models, crucial to the study of the Gibbs measure,
are widely believed to behave similarly to those of the models
above, at least for models that exhibit only finite replica symmetry
breaking. Theorem \ref{thm:ext-proc} verifies this on the level of
extremal processes: the extremal process of the spherical $p$-spin
model, like all the above mentioned models, belongs to the class of
SDPPPs of exponential intensity. Strikingly, it does not even have
a shift or decoration - exactly like REM, the simplest of all models.

For the BBM, BRW, and DGFF, the proof of convergence of the extremal
process is in a sense the culmination of a long, incremental progress
spanning over decades. Earlier works were limited to the study of
the global maximum and concerned topics like the law of large numbers
\cite{BoltDeuschGiac,Mckean}, second order corrections \cite{Bramson1,Bramson2,BZ,McDiarmid},
tail estimates \cite{BDZ2,Ding,DZ}, tightness \cite{Addario-BerryReed,Bachmann,BDZ,BZ2,BZ,HuShi},
and convergence and limiting law \cite{Aidekon,BDZ2,LalSel}. In addition,
estimates on the number of extremal points \cite{Daviaud,DZ} and
their genealogical relations \cite{ABKGene,BDZ2} were obtained. The
mathematics behind this is rich and beautiful and has produced vast
literature. We emphasize that the analysis of the extremal processes
undertaken in the papers mentioned above heavily relied on those prior
results. In light of the above, we find it rather remarkable that
for the spherical $p$-spin model all we need for our proof, as we
explain below, are moment computations combined with local estimates
of $H_{N}$. 

Lastly, we mention the class of log-correlated Gaussian fields. The
extremal process of those is also conjectured to be an SDPPP. There
are two supporting facts to this conjecture. First, the maximum of a wide class
of log-correlated fields is known to have a shifted Gumbel law \cite{DRZ,MadauleMax}.
Secondly, log-correlated fields are believed to satisfy the so-called
`freezing' phenomenon \cite{CLD} which is intimately related to the
structure of SDPPP \cite{Freezing}.

\subsection*{Method of proof: moments and invariance}

Two of the main tools in our analysis are statements about the first and second moments
of the number of critical points in an interval. In addition to (\ref{eq:1st_mom}),
Auffinger, Ben Arous and {\v{C}}ern{\'y}  \cite{A-BA-C} computed
asymptotics of the first moment on a finer scale (see Theorem 2.17
there). Their proof was based on combining Plancherel-Rotach type asymptotics  satisfied by the Hermite polyomials with a general formula they derived for the first moment of the number of critical points. Using a similar method we prove the following.
\begin{prop}
\label{prop:intensity}Let $p\geq3$. Suppose $J_{N}\subset\mathbb{R}$
is a sequence of non-degenerate intervals such that $J_{N}\subset\left(-a_{N},a_{N}\right)$
for some sequence satisfying $a_{N}/N\to0$. Then the intensity measure
of $\xi_{N}$ has density $\nu_{N}$ (w.r.t to the Lebesgue measure),
such that, as $N\to\infty$,
\begin{equation}
\nu_{N}\left(x_{N}\right)=e^{c_{p}x_{N}}\cdot\left(1+o\left(1\right)\right),\label{eq:nu}
\end{equation}
uniformly for any sequence $x_{N}\in J_{N}$, where $c_p$ is as in \eqref{eq:0406-1}. In particular,  with
$x_{N}=x$, the above holds uniformly in $x$ on compacts.
\end{prop}
By a direct adaptation of the main results of \cite{2nd} we have
the following.\footnote{
	For fixed open $J_N=J$, by the Portmanteau theorem 
	$\liminf_{N\to\infty} \mathbb E\left\{ \left(\xi_{N}\left(J\right)\right)^{2}\right\}\geq \mathbb E\left\{ \left(\xi_{\infty}\left(J\right)\right)^{2}\right\} $ where $\xi_{\infty}$ is the limiting process of Theorem \ref{thm:ext-proc}. Proposition \ref{prop:intensity} (and e.g. \cite[Theorem 4.2]{KallenbergMeas}) implies that $\lim_{N\to\infty} \mathbb E\left\{ \xi_{N}\left(J\right)\right\}= \mathbb E\left\{ \xi_{\infty}\left(J\right)\right\}$. Therefore, for such $J_N=J$, it in fact follows from the convergence to Poisson process of Theorem \ref{thm:ext-proc} that	
	\eqref{eq:58} below holds with lim instead of limsup and with equality. We note that the proof of Theorem \ref{thm:ext-proc} only requires the upper bound stated in Proposition \ref{prop:2nd moment}.} 
\begin{prop}
\label{prop:2nd moment}Let $p\geq3$. Suppose $J_{N}\subset\mathbb{R}$
is as in Proposition \ref{prop:intensity} and assume (for simplicity)
that $\int_{J_{N}}e^{c_{p}x}dx\geq c$ for any $N$ for some constant
$c>0$. Then
\begin{equation}
\limsup_{N\to\infty}\frac{\mathbb{E}\left\{ \left(\xi_{N}\left(J_{N}\right)\right)^{2}-\xi_{N}\left(J_{N}\right)\right\} }{\left(\mathbb{E}\left\{ \xi_{N}\left(J_{N}\right)\right\} \right)^{2}}\leq 1.\label{eq:58}
\end{equation}
In particular, the above holds with a fixed non-degenerate interval
$J_{N}=J$.
\end{prop}
An immediate consequence of Proposition \ref{prop:intensity} is that
the sequence of extremal processes $\left\{ \xi_{N}\right\} _{N=1}^{\infty}$
is relatively compact in the vague topology (see for example, \cite[Lemma 4.5]{KallenbergMeas}).
Thus it converges in distribution to $\xi_{\infty}$ if and only if
$\xi_{\infty}$ is its unique subsequential limit. 

The method we use in order to prove the latter convergence is inspired
by a beautiful argument of Biskup and Louidor \cite{Louidor} originally
introduced in the context of the DGFF: they identify the law of subsequential
limits by invoking a result of Liggett \cite{Liggett} which characterizes
the probability laws of point processes that are invariant under an
evolution of the atoms by independent Markov chains of the same common
law. In order to introduce such evolution, they perturb the Gaussian
field by an independent `increment' identical in law to the original
field, up to scaling. We now sketch how this works in our setting.

Let $H_{N}^{\prime}\left(\boldsymbol{\sigma}\right)$ be an independent
copy of $H_{N}\left(\boldsymbol{\sigma}\right)$, set
\begin{equation}
H_{N}^{+}\left(\boldsymbol{\sigma}\right)=H_{N}\left(\boldsymbol{\sigma}\right)+\sqrt{\frac{1}{N}}H_{N}^{\prime}\left(\boldsymbol{\sigma}\right),\label{eq:H+}
\end{equation}
and note that, with $\stackrel{d}{=}$ denoting equality in distribution,
\begin{equation}
H_{N}^{+}\left(\boldsymbol{\sigma}\right)\stackrel{d}{=}s_NH_{N}\left(\boldsymbol{\sigma}\right),\quad s_N=\sqrt{\frac{N+1}{N}}.\label{eq:34-1}
\end{equation}

Define the extremal process of $H_{N}^{+}$ similarly to (\ref{eq:xi_N}),
without changing the centering term $m_{N}$. From (\ref{eq:34-1})
one can see that, for a convergent subsequence, the limiting extremal
process associated to $H_{N}^{+}$ is the same as that of $H_{N}$
up to a deterministic shift (of $-\frac{1}{2}E_{0}$, see (\ref{eq:29-1})).
On the other hand, (\ref{eq:H+}) can be exploited to study the effect
of the perturbation $\sqrt{1/N}H_{N}^{\prime}\left(\boldsymbol{\sigma}\right)$
on the same limiting extremal process. For each critical\footnote{As we shall see, the only critical points that will be relevant are
minimum points.} point $\boldsymbol{\sigma}$ of $H_{N}$, using a quadratic approximation
for $H_{N}$ and a linear approximation for $H_{N}^{\prime}$ in a
small neighborhood of $\boldsymbol{\sigma}$, on an appropriate event,
we will show that the neighborhood contains a (single) critical point of
$H_{N}^{+}\left(\boldsymbol{\sigma}\right)$ and derive estimates
for its location and critical value. The perturbation to each critical
value will be proved to be equal, up to a lower order term, to $\sqrt{\frac{1}{N}}H_{N}^{\prime}\left(\boldsymbol{\sigma}\right)-C_{0}$,
for an appropriate constant $C_{0}>0$. From the results of \cite{2nd},
the maximal overlap $R\left(\boldsymbol{\sigma},\boldsymbol{\sigma}'\right)$
(see (\ref{eq:overlap})) of any two  critical points $\boldsymbol{\sigma}\neq\pm\boldsymbol{\sigma}'$
with values less than $m_{N}+x$, converges in probability to $0$.
This means that the values of $H_{N}^{\prime}$ at those points are
essentially independent. By showing in addition that the perturbed
critical points `cover' all critical points of $H_{N}^{+}$ we will obtain
the following.
\begin{thm}
\label{thm:subseq}Suppose that $\xi_{N_{k}}$ is a convergent subsequence
of $\xi_{N}$ with limit in distribution $\bar{\xi}_{\infty}=\sum\delta_{\bar{\xi}_{\infty,i}}$,
then the corresponding subsequence $\xi_{N_{k}}^{+}$ of $\xi_{N}^{+}$
converges in distribution to $\sum\delta_{\bar{\xi}_{\infty,i}+W_{i}-C_{0}}$,
where $W_{i}\sim N\left(0,1\right)$ are i.i.d and independent of
$\bar{\xi}_{\infty}$ and $C_{0}$ is given in (\ref{eq:C0}). Moreover,
$\sum\delta_{\bar{\xi}_{\infty,i}+W_{i}-C_{0}}$ is locally finite,
i.e. a point process.
\end{thm}
By Liggett's characterization \cite{Liggett} this implies that any
subsequential limit in distribution $\bar{\xi}_{\infty}\sim PPP\left(\bar{\mu}\right)$
is a Cox process, namely $\bar{\mu}$ is random, such that $\bar{\mu}$
has density w.r.t the Lebesgue measure which satisfies almost surely
\begin{equation}
\frac{d\bar{\mu}}{dx}\in\left\{ \alpha\left(x\right)\equiv a:\,a\geq0\right\} \cup\left\{ \alpha\left(x\right)=e^{c_{p}\left(x-z\right)}:\,z\in\mathbb{R}\right\} ,\label{eq:densities}
\end{equation}
where we recall that $c_{p}\triangleq\Theta_{p}^{\prime}\left(-E_{0}\right)$.
Finally, to show that indeed we must have $\frac{d\bar{\mu}}{dx}=e^{c_{p}x}$,
we shall use Propositions \ref{prop:intensity} and \ref{prop:2nd moment}.

We finish by commenting that, besides Liggett's \cite{Liggett}, various invariance properties were discovered in the context of extremal process and spin glass theory: the quasi-stationarity properties of competing particle systems studied by Ruzmaikina and Aizenman \cite{RuzmaikinaAizenman}, Arguin and Aizenman \cite{ArguinAizenman} and Arguin \cite{Arguin}; the Bolthausen-Sznitman invariance property \cite{Bolthausen1}; the Ghirlanda-Guerra identities \cite{GG} and the stochastic stability  of Aizenman and Contucci \cite{Aizenman1998} - see also \cite{Panchenko2012} by Panchenko for a unification of the two; and lastly, exponential stability and the freezing phenomenon studied by Maillard \cite{Maillard} and the authors \cite{Freezing}, respectively.

\subsection*{Structure of the paper}

After introducing notation in the next section, we prove Theorem \ref{thm:ext-proc}
and Corollary \ref{cor:GS} in Section \ref{sec:pfThm1}, assuming
Propositions \ref{prop:intensity} and \ref{prop:2nd moment} and
Theorem \ref{thm:subseq}. Sections \ref{sec:prop-intensity} and
\ref{sec:prop-2nd} are devoted to the proofs of Propositions \ref{prop:intensity}
and \ref{prop:2nd moment}, respectively. In Section \ref{sec:auxres}  we state and prove several auxiliary results and prove Theorem \ref{thm:subseq} using them. The most important of those auxiliary results is Lemma \ref{lem:5}. Its proof is given in Section \ref{sec:pfProp5} where the linear and quadratic approximations mentioned above are defined and investigated and several related tools are discussed.

\section{\label{sec:notation}Notation }

In this section we introduce some notation needed in the sequel. The covariance function of $H_{N}$ is given by 
\begin{equation}
\mathbb E\left\{H_N(\boldsymbol{\sigma}),H_N(\boldsymbol{\sigma}')\right\}=N\left(R\left(\boldsymbol{\sigma},\boldsymbol{\sigma}'\right)\right)^{p}\triangleq N\left(\frac{\left\langle \boldsymbol{\sigma},\boldsymbol{\sigma}'\right\rangle }{N}\right)^{p},\label{eq:overlap}
\end{equation}
where $\left\langle \cdot,\cdot\right\rangle $ denotes the standard
inner product and $R\left(\cdot,\cdot\right)$ is the so-called overlap
function. Note that since $H_{N}$ is a centered Gaussian field, the covariance function
characterizes it. 

A random matrix $\mathbf{M}_{N}$ from the (normalized) $N\times N$ Gaussian orthogonal
ensemble, or an $N\times N$ GOE matrix, for short, is a real, symmetric
matrix such that all elements are centered Gaussian variables which,
up to symmetry, are independent with variance given by
\[
\mathbb{E}\left\{ \mathbf{M}_{N,ij}^2\right\} =\begin{cases}
1/N, & \,i\neq j\\
2/N, & \,i=j.
\end{cases}
\]

Denote by 
\begin{equation}
\omega_{N}=\frac{2\pi^{N/2}}{\Gamma\left(N/2\right)}\label{eq:omega_vol}
\end{equation}
 the surface area of the $N-1$-dimensional unit sphere.

Let $\mu^{*}$ denote the semicircle measure, the density of which
with respect to the Lebesgue measure is 
\begin{equation}
\frac{d\mu^{*}}{dx}=\frac{1}{2\pi}\sqrt{4-x^{2}}\mathbf{1}_{\left|x\right|\leq2},\label{eq:semicirc}
\end{equation}
and define the function (see, e.g., \cite[Proposition II.1.2]{logpotential})
\begin{align}
\Omega(x) & \triangleq\int_{\mathbb{R}}\log\left|\lambda-x\right|d\mu^{*}\left(\lambda\right)\label{eq:Omega}\\
 & =\begin{cases}
\frac{x^{2}}{4}-\frac{1}{2}, & \mbox{ if }0\leq\left|x\right|\leq2,\\
\frac{x^{2}}{4}-\frac{1}{2}-\left[\frac{\left|x\right|}{4}\sqrt{x^{2}-4}-\log\left(\sqrt{\frac{x^{2}}{4}-1}+\frac{\left|x\right|}{2}\right)\right], & \mbox{ if }\left|x\right|>2.
\end{cases}\nonumber 
\end{align}
The exponential growth rate function of (\ref{eq:1st_mom}) is given
in \cite{A-BA-C} by 
\begin{equation}
\Theta_{p}\left(u\right)=\begin{cases}
\frac{1}{2}+\frac{1}{2}\log\left(p-1\right)-\frac{u^{2}}{2}+\Omega\left(\gamma_{p}u\right), & \mbox{ if }u<0,\\
\frac{1}{2}\log\left(p-1\right), & \mbox{ if }u\geq0,
\end{cases}\label{eq:Theta_p}
\end{equation}
 where $\gamma_{p}=\sqrt{p/\left(p-1\right)}$. Define the constants
\begin{align}
C_{0} & =\frac{1}{2}\gamma_{p}\int\frac{1}{\gamma_{p}E_{0}-x}d\mu^{*}\left(x\right),\label{eq:C0}\\
K_{0} & =\frac{1}{2}E_{0}-\frac{1}{c_{p}}\log\left(\frac{1+\iota_{p}}{2\sqrt{2\pi}}\tilde{h}\left(-\frac{1}{2}\gamma_{p}E_{0}\right)\right),\nonumber 
\end{align}
where
\begin{equation}
\label{htilde}\tilde{h}\left(x\right)=\left(\left|\frac{x-1}{x+1}\right|^{1/4}+\left|\frac{x+1}{x-1}\right|^{1/4}\right).
\end{equation}
We note that, by a straightforward calculation, 
\begin{equation}
C_{0}=\frac{1}{2}E_{0}-\frac{1}{2}c_{p}.\label{eq:C0E0rel}
\end{equation}

\section{\label{sec:pfThm1}Proof of Theorem \ref{thm:ext-proc} and Corollary
\ref{cor:GS}, assuming Propositions \ref{prop:intensity} and \ref{prop:2nd moment}
and Theorem \ref{thm:subseq}}

As argued in the introduction, by Proposition \ref{prop:intensity}
the sequence $\left\{ \xi_{N}\right\} _{N=1}^{\infty}$ is relatively
compact (see for example, \cite[Lemma 4.5]{KallenbergMeas}) and all
that we need to show is that $\xi_{\infty}$ (see (\ref{eq:xi_infty}))
is its unique subsequential limit. We define the extremal point process
associated to $H_{N}^{+}$ by
\[
\xi_{N}^{+}=\left(1+\iota_{p}\right)^{-1}\sum_{\boldsymbol{\sigma}\in\mathscr{C}_{N}^{+}}\delta_{H_{N}^{+}\left(\boldsymbol{\sigma}\right)-m_{N}},
\]
where $\mathscr{C}_{N}^{+}$ is defined similarly to $\mathscr{C}_{N}$.
Let $\xi_{N_{k}}$ be an arbitrary convergent subsequence of $\xi_{N}$
with limit in distribution $\bar{\xi}_{\infty}$ and let $\xi_{N_{k}}^{+}$
be the corresponding subsequence of $\xi_{N}^{+}$. Let $\mathcal{T}_{x}$
denote the shift operator, $\mathcal{T}_{x}\eta\left(\cdot\right)=\eta\left(\cdot-x\right)$,
and $\mathcal{S}_{x}$ the scaling operator $\mathcal{S}_{x}\eta\left(\cdot\right)=\eta\left(\frac{1}{x}\,\cdot\,\right)$.
From the fact that 
\[
\xi_{N_{k}}^{+}\stackrel{d}{=}\mathcal{T}_{-m_{N}}\circ\mathcal{S}_{s_{N}}\circ\mathcal{T}_{m_{N}}\xi_{N_{k}},
\]
with $s_{N}=\sqrt{\left(N+1\right)/N}$, it follows that $\xi_{N_{k}}^{+}$
is also convergent and, since $\lim_{N\to\infty}m_{N}\left(s_{N}-1\right)=-\frac{1}{2}E_{0}$,
\begin{equation}
\lim_{N\to\infty}\xi_{N_{k}}^{+}=\mathcal{T}_{-\frac{1}{2}E_{0}}\bar{\xi}_{\infty}.\label{eq:29-1}
\end{equation}
Combined with Theorem \ref{thm:subseq} this yields, assuming $\bar{\xi}_{\infty}=\sum\delta_{\bar{\xi}_{\infty,i}}$,
\[
\sum\delta_{\bar{\xi}_{\infty,i}}\stackrel{d}{=}\sum\delta_{\bar{\xi}_{\infty,i}+W_{i}-C_{0}+\frac{1}{2}E_{0}}.
\]

By Liggett's characterization \cite[Corollary 4.10]{Liggett}, $\bar{\xi}_{\infty}$
is a Cox process such that the random intensity measure, which we
denote by $\bar{\mu}$, satisfies a.s.
\[
\bar{\mu}=\bar{\mu}\star\mathbb{P}\left\{ W_{i}+\frac{1}{2}c_{p}\in\cdot\right\} ,
\]
where we used (\ref{eq:C0E0rel}) and $\star$ is the convolution
operation. By the Choquet-Deny Theorem \cite[Theorem 3']{Deny}, a.s.,
$\mu$ has density w.r.t the Lebesgue measure which belongs to the set
of functions in (\ref{eq:densities}). 

Now, define the intervals $A_{n}=\left(n,n+1\right)$ and set $X_{n,N_{k}}=\xi_{N_{k}}\left(A_{n}\right)/\mathbb{E}\left\{ \xi_{N_{k}}\left(A_{n}\right)\right\} $.
Note that since $\xi_{N_{k}}\stackrel[k\to\infty]{d}{\to}\bar{\xi}_{\infty}$,
we have that\footnote{See \cite[Theorem 4.2]{KallenbergMeas}, and note that $A_{n}$ are
continuity sets of $\bar{\xi}_{\infty}$.} 
\begin{equation}
\label{eq:0406-2} X_{n,N_{k}}\stackrel[{\scriptstyle k\to\infty}]{d}{\to}X_{n}\triangleq\bar{\xi}_{\infty}\left(A_{n}\right)/C_{n},
\end{equation}
for any $n\geq1$, where 
\[
C_{n}=\frac{1}{c_{p}}e^{c_{p}n}\left(e^{c_{p}}-1\right)
\]
is the limit, as $k\to\infty$, of the expectations $\mathbb{E}\left\{ \xi_{N_{k}}\left(A_{n}\right)\right\} $,
computed using Proposition \ref{prop:intensity}. From Chebyshev's
inequality, 
\begin{align*}
\limsup_{n\to\infty}\mathbb{P}\left\{ \left|X_{n}-1\right|\geq\epsilon\right\}  & =\limsup_{n\to\infty}\lim_{k\to\infty}\mathbb{P}\left\{ \left|X_{n,N_{k}}-1\right|\geq\epsilon\right\} \leq\epsilon^{-2}\limsup_{n\to\infty}\limsup_{k\to\infty}\mathbb{E}\left\{ X_{n,N_{k}}^{2}-1\right\} .
\end{align*}
By Proposition \ref{prop:intensity} and Proposition \ref{prop:2nd moment} and the fact that $\mathbb E\{(\xi_{N_k}(A_n))^2\}\geq (\mathbb E\{\xi_{N_k}(A_n)\})^2$
\[
\lim_{n\to\infty}\lim_{k\to\infty}\mathbb{E}\left\{ \xi_{N_{k}}\left(A_{n}\right)\right\} =\infty\mbox{\,\,\ and\,\,}\lim_{n\to\infty}\limsup_{k\to\infty}\mathbb{E}\left\{ X_{n,N_{k}}^{2}\right\} =1.
\]
Therefore $X_{n}$ converges in probability to $1$, as $n\to\infty$.
It can be verified that if $\zeta\sim PPP\left(\alpha\left(x\right)dx\right)$
with $\alpha$ being a function in the set (\ref{eq:densities}),
then $\zeta\left(A_{n}\right)/C_{n}\to1$ in probability only if $\alpha\left(x\right)=e^{c_{p}x}$.
It follows that $d\bar{\mu}/dx$ must be of this exact form, which
completes the proof of Theorem \ref{thm:ext-proc}.

Corollary \ref{cor:GS} follows directly from Theorem \ref{thm:ext-proc}
if we show that 
\begin{equation}
\left\{ \min_{\boldsymbol{\sigma}\in\mathbb{S}^{N-1}\left(\sqrt{N}\right)}H_{N}\left(\boldsymbol{\sigma}\right)-m_{N}\right\} _{N\geq1}\label{eq:min}
\end{equation}
is uniformly tight. From \eqref{eq:0406-2} and Theorem \ref{thm:ext-proc}, $X_{n,N}\to X_n$ in distribution as $N\to\infty$. Thus, since $X_n\to 1$ in probability as $n\to\infty$,
\[
\lim_{x\to\infty}\limsup_{N\to\infty}\mathbb{P}\left\{ \min_{\boldsymbol{\sigma}\in\mathbb{S}^{N-1}\left(\sqrt{N}\right)}H_{N}\left(\boldsymbol{\sigma}\right)-m_{N}\geq x\right\} \leq\lim_{n\to\infty}\lim_{N\to\infty}\mathbb{P}\left\{ X_{n,N}=0\right\} =0.
\]

From \eqref{eq:1st_mom} and the fact that $\Theta_{p}\left(x\right)<0$
for $x<-E_{0}$, Markov's inequality, and the Borell-TIS inequality
\cite[Theorem 2.1.1]{RFG}, we have that for some sequence $a_{N}$
satisfying $a_{N}/N\to0$,
\[
\limsup_{N\to\infty}\mathbb{P}\left\{ \min_{\boldsymbol{\sigma}\in\mathbb{S}^{N-1}\left(\sqrt{N}\right)}H_{N}\left(\boldsymbol{\sigma}\right)-m_{N}\leq a_{N}\right\} =0.
\]
By Proposition \ref{prop:intensity} and Markov's inequality,
\[
\lim_{x\to-\infty}\limsup_{N\to\infty}\mathbb{P}\left\{ a_{N}\leq\min_{\boldsymbol{\sigma}\in\mathbb{S}^{N-1}\left(\sqrt{N}\right)}H_{N}\left(\boldsymbol{\sigma}\right)-m_{N}\leq x\right\} \leq\lim_{x\to-\infty}\frac{1}{c_{p}}e^{c_{p}x}.
\]
 Hence (\ref{eq:min}) is uniformly tight and the proof of Corollary
\ref{cor:GS} is completed.\qed

\section{\label{sec:prop-intensity}Proof of Proposition \ref{prop:intensity}}

Suppose $J_{N}$ is a sequence of intervals as in the statement of
the proposition and set $D_{N}=J_{N}+m_{N}$. An application of the
Kac-Rice formula \cite[Theorem 12.1.1]{RFG} yields an integral formula
of the form 
\[
\mathbb{E}\big\{ \# \left\{\boldsymbol{\sigma}\in\mathscr{C}_{N}: H_N(\boldsymbol{\sigma})\in D_{N}\right\}\big\} =\int_{D_{N}}\rho_{N}\left(u\right)du.
\]
This has been worked out in \cite[Lemmas 3.1]{A-BA-C}, together with
a computation of certain related conditional laws \cite[Lemmas 3.2]{A-BA-C},
from which 
\begin{equation}
\rho_{N}\left(u\right)\triangleq\omega_{N}\left(\frac{p-1}{2\pi}\left(N-1\right)\right)^{\frac{N-1}{2}}\frac{1}{\sqrt{2\pi N}}e^{-\frac{1}{2N}u^{2}}\mathbb{E}\left\{ \left|\det\left(\mathbf{M}-\frac{\gamma_{p}}{\sqrt{N\left(N-1\right)}}u\mathbf{I}\right)\right|\right\} ,\label{eq:49}
\end{equation}
where $\mathbf{M}$ is a GOE matrix and $\mathbf{I}$ is the identity
matrix both of dimension $N-1$.

The following lemma is a particular case of \cite[Lemma 21]{2nd}.
\begin{lem}
\label{lem:19}Let $\mathbf{M}$ be a GOE matrix of dimension $N-1$ and
let $t_{2}>t_{1}>2$. Then for any $\delta>0$ there exists $c=c(\delta)>0$ such that, for
large enough $N$, uniformly in $v\in\left(t_{1},t_{2}\right)$,
\begin{equation}
\mathbb{E}\left\{ \left|\det\left(\mathbf{M}+v\mathbf{I}\right)\right|\mathbf{1}\left\{ \lambda_{min}<t_{1}-2-\delta\right\} \right\} \leq e^{-cN}\mathbb{E}\left\{ \left|\det\left(\mathbf{M}+v\mathbf{I}\right)\right|\right\} ,\label{eq:48}
\end{equation}
where $\lambda_{min}$ denotes the minimal eigenvalue of $\mathbf{M}+v\mathbf{I}$. Therefore, uniformly in $v\in\left(t_{1},t_{2}\right)$,
as $N\to\infty$,
\[
\mathbb{E}\left\{ \left|\det\left(\mathbf{M}+v\mathbf{I}\right)\right| \right\} =(1+o(1)) \mathbb{E}\left\{ \det\left(\mathbf{M}+v\mathbf{I}\right)\right\}.
\]
\end{lem}

From our assumption on $J_{N}$, there exist $t_{2}>t_{1}>2$ such
that for large enough $N$ and for any $u\in D_{N}$, $-\frac{\gamma_{p}}{\sqrt{N\left(N-1\right)}}u\in\left(t_{1},t_{2}\right)$.
Therefore  (\ref{eq:49}) still holds, uniformly in $u\in D_{N}$,
if we remove the absolute value and multiply the expectation by $\left(1+o\left(1\right)\right)$,
as $N\to\infty$.

Let $p_{N}\left(x\right)$ be the Hermite polynomials, with the normalization
\[
\int p_{N}\left(x\right)p_{M}\left(x\right)e^{-x^{2}}dx=\delta_{NM}.
\]
Corollary 11.6.3 of \cite{RFG} states that
\begin{equation}
\left(N-1\right)^{\frac{N-1}{2}}\mathbb{E}\left\{ \det\left(\mathbf{M}-v\mathbf{I}\right)\right\} =\left(-1\right)^{N-1}\pi^{1/4}\sqrt{\left(N-1\right)!}p_{N-1}\left(\sqrt{\frac{N-1}{2}}v\right).\label{eq:51}
\end{equation}
From \cite[Theorem 2.2]{DKMVZ99} (see \cite[p. 20]{DeiftGioev} for
a statement more convenient for our needs; in our case, $V\left(x\right)=x^{2}$,
$V_{N}\left(x\right)=x^{2}/2$, $c_{N}=\sqrt{2N}$, $d_{N}=0$, $h_{N}\left(x\right)=4$,
$m=1$) and some calculus, $p_{N}\left(x\right)$ satisfy the following
Plancherel-Rotach type asymptotics. For any $\delta>0$, uniformly
in $x<-\left(1+\delta\right)$, as $N\to\infty$,
\begin{align}
p_{N}\left(\sqrt{2N}x\right) & =\frac{\left(-1\right)^{N-1}}{\sqrt{4\pi\sqrt{2N}}}\exp\left\{ N\left(\Omega\left(2x\right)+\frac{1}{2}\right)\right\} \tilde{h}\left(x\right)\left(1+O\left(\frac{1}{N}\right)\right),\label{eq:53}
\end{align}
where $\Omega(x)$ and $\tilde{h}(x)$  are defined in \eqref{eq:Omega} and \eqref{htilde}, respectively.
Note that, for small enough $\delta>0$, for large enough $N$, if
$u\in D_{N}$, then 
\[
x_{u}=x_{u,N,p}\triangleq\frac{1}{2}\cdot\frac{\gamma_{p}}{\sqrt{N\left(N-1\right)}}u<-\left(1+\delta\right).
\]
Therefore, combining (\ref{eq:51}) and (\ref{eq:53}), we have that
\begin{align*}
\rho_{N}\left(u\right) & =\omega_{N}\left(\frac{p-1}{2\pi}\right)^{\frac{N-1}{2}}\frac{1}{\sqrt{2\pi N}}e^{-\frac{1}{2N}u^{2}}\pi^{1/4}\sqrt{\left(N-1\right)!}p_{N-1}\left(\sqrt{2\left(N-1\right)}x_{u}\right)\left(1+o\left(1\right)\right)\\
 & =\frac{1}{2\sqrt{2\pi}}\tilde{h}\left(x_{u}\right)\exp\left\{ \left(N-1\right)\Theta_{p}\left(\frac{u}{\sqrt{N\left(N-1\right)}}\right)-\frac{1}{2}\log N\right\} \left(1+o\left(1\right)\right),
\end{align*}
uniformly in $u\in D_{N}$, where $\Theta_p(x)$ is defined in \eqref{eq:Theta_p}.

Using a first order Taylor expansion for $\Theta_{p}$ around $-E_{0}$
(recall that $\Theta_{p}\left(-E_{0}\right)=0$ and $\Theta_p'(-E_0)=c_p$) and the fact that
\[
\frac{1}{\sqrt{\left(N-1\right)}}=\frac{1}{\sqrt{N}}\left(1+\frac{1}{2N}+O\left(\frac{1}{N^{2}}\right)\right),
\]
we have, uniformly in $u\in D_{N}$, 
\begin{align}
\rho_{N}\left(u\right) & =\frac{1}{2\sqrt{2\pi}}\tilde{h}\left(-\frac{1}{2}\gamma_{p}E_{0}\right)\exp\left\{ c_{p}\left(\left(u+E_{0}N\right)-\frac{1}{2}E_{0}\right)-\frac{1}{2}\log N\right\} \left(1+o\left(1\right)\right).\label{eq:52}
\end{align}

Now, note that, with $\nu_N(x)$ denoting the density function of $\xi_N$,
\[
\nu_{N}\left(x\right)=\left(1+\iota_{p}\right)^{-1}\rho_{N}\left(x+m_{N}\right).
\]
Substitution of $m_{N}$ and $K_{0}$, defined in (\ref{eq:m_N})
and (\ref{eq:C0}), in (\ref{eq:52}) completes the proof of Proposition
\ref{prop:intensity}.\qed

\section{\label{sec:prop-2nd}Proof of Proposition \ref{prop:2nd moment}}

The proposition is a direct consequence of the results of \cite{2nd}.
With $B\subset\mathbb R$ and $I_R\subset [-1,1]$, in \cite{2nd} the random variables $\mbox{Crt}_{N}\left(B\right)$
and $\left[\mbox{Crt}_{N}\left(B,I_{R}\right)\right]_{2}$ are defined
as follows. The former is the number of points $\boldsymbol{\sigma}$
on the sphere $\mathbb{S}^{N-1}\left(\sqrt{N}\right)$ such that $\boldsymbol{\sigma}$
is a critical point of $H_{N}$, with critical value $H_{N}\left(\boldsymbol{\sigma}\right)\in NB$.
The latter is the number of ordered pairs $\left(\boldsymbol{\sigma}_{1},\boldsymbol{\sigma}_{2}\right)$
of points on the sphere $\mathbb{S}^{N-1}\left(\sqrt{N}\right)$ such
that $\boldsymbol{\sigma}_{i}$ are critical points of $H_{N}$, with
critical values $H_{N}\left(\boldsymbol{\sigma}_{i}\right)\in NB$,
and the overlap $R\left(\boldsymbol{\sigma}_{1},\boldsymbol{\sigma}_{2}\right)$
belongs to $I_{R}$. Recall that for odd $p$, $H_{N}\left(-\boldsymbol{\sigma}\right)=-H_{N}\left(\boldsymbol{\sigma}\right)$,
and for even $p$, $H_{N}\left(-\boldsymbol{\sigma}\right)=H_{N}\left(\boldsymbol{\sigma}\right)$,
for any $\boldsymbol{\sigma}$. Also, for any $B\subset\left(-\infty,0\right)$,
clearly $NB\cap\left(-NB\right)$ is empty. From these facts we have that for such $B$, 
\begin{equation}
\left[\mbox{Crt}_{N}\left(B,\left(-1,1\right)\right)\right]_{2}=\left(\mbox{Crt}_{N}\left(B\right)\right)^{2}-\left(1+\iota_{p}\right)\mbox{Crt}_{N}\left(B\right).\label{eq:59}
\end{equation}

For any small enough $\delta,\epsilon>0$, denoting $B_{\delta}=\left(-E_{0}-\delta,-E_{0}+\delta\right)$
and $I_{\epsilon}=\left(-1,1\right)\setminus\left(-\epsilon,\epsilon\right)$,
Theorem 5 and Lemma 6 of \cite{2nd}  yield
\begin{equation}
\lim_{N\to\infty}\frac{1}{N}\log\left(\mathbb{E}\left\{ \left[\mbox{Crt}_{N}\left(B_{\delta},I_{\epsilon}\right)\right]_{2}\right\} \right)\leq\sup_{r\in I_{\epsilon}}\sup_{u\in B_{\delta}}\Psi_{p}\left(r,u\right),\label{eq:57}
\end{equation}
where $\Psi_{p}\left(r,u\right)=\Psi_{p}\left(r,u,u\right)$ and $\Psi_{p}\left(r,u_{1},u_{2}\right)$
is given in equation (2.7) of \cite{2nd}. By part (1) of Lemma 7 of \cite{2nd}, 
\[
\sup_{r\in I_{\epsilon}}\Psi_{p}\left(r,-E_0\right)<\Psi_{p}\left(0,-E_0\right)=2\Theta_p(-E_0)=0.
\]

The function  $\Psi_{p}(r,u)$ is continuous on $(-1,1)\times\mathbb R$ and can be extended to a continuous function on $[-1,1]\times\mathbb R$. Thus, with $\epsilon$ fixed, for small enough $\delta>0$ we also have that 
\[
\sup_{r\in I_{\epsilon}}\sup_{u\in B_{\delta}}\Psi_{p}\left(r,-E_0\right)<\Psi_{p}\left(0,-E_0\right)=2\Theta_p(-E_0)=0,
\]
and the left-hand side of (\ref{eq:57}) is negative.

Let $J_{N}$ be a sequence of intervals as in the statement of the
proposition, set $A\left(N\right)=\left(J_{N}+m_{N}\right)/N$ and note that
\[
\mathbb{E}\left\{ \xi_{N}\left(J_{N}\right)\right\} =\left(1+\iota_{p}\right)^{-1}\mathbb{E}\left\{ \mbox{Crt}_{N}\left(A\left(N\right)\right)\right\} .
\]
By (\ref{eq:59}),
for any $\epsilon>0$ and small enough $\delta>0$, for large enough
$N$, $A\left(N\right)\subset B_{\delta}$ and
\begin{align}
\mathbb{E}\left\{ \left(\xi_{N}\left(J_{N}\right)\right)^{2}-\xi_{N}\left(J_{N}\right)\right\}  & =\left(1+\iota_{p}\right)^{-2}\mathbb{E}\left\{ \left[\mbox{Crt}_{N}\left(A\left(N\right),\left(-1,1\right)\right)\right]_{2}\right\} ,\nonumber \\ 
\mathbb{E}\left\{ \left[\mbox{Crt}_{N}\left(A\left(N\right),I_{\epsilon}\right)\right]_{2}\right\}  & \leq\mathbb{E}\left\{ \left[\mbox{Crt}_{N}\left(B_{\delta},I_{\epsilon}\right)\right]_{2}\right\} .
\end{align}
From our assumption on $J_{N}$ and Proposition \ref{prop:intensity},
for large enough $N$, $\mathbb{E}\left\{ \mbox{Crt}_{N}\left(A\left(N\right)\right)\right\} >c/2$.
Thus, since the left-hand side of (\ref{eq:57}) is negative, the
proposition will follow if we can show that
\begin{equation}
\lim_{N\to\infty}\frac{\mathbb{E}\left\{ \left[\mbox{Crt}_{N}\left(A\left(N\right),\left(-\epsilon_N,\epsilon_N\right)\right)\right]_{2}\right\} }{\left(\mathbb{E}\left\{ \mbox{Crt}_{N}\left(A\left(N\right)\right)\right\} \right)^{2}}\leq 1,\label{eq:61}
\end{equation}
for any sequence $\epsilon_N>0$ such that $\epsilon_N\to 0$ as $N\to\infty$.
This follows from Lemma 19 of \cite{2nd}.

\section{\label{sec:auxres}Auxiliary results and the proof of Theorem \ref{thm:subseq}}

We begin with four lemmas which will be used in the proof
of Theorem \ref{thm:subseq}. In the proof of the latter 
we shall relate the (sequence of) extremal processes $\xi_{N}^{+}$
of the perturbed Hamiltonian $H_{N}^{+}$ to those of $H_{N}$, namely
$\xi_{N}$. We will show that they, in some sense, asymptotically
differ by adding independent increments to the corresponding atoms,
at least on bounded intervals around $m_{N}$. The approximation Lemma \ref{lem:apx.Ligg}
will be used to conclude from those relations - concerning the restrictions
of non-limiting processes to bounded intervals - a similar relation
between the limiting process $\xi_{\infty}$ and $\xi_{\infty}^{+}$.
Its proof is simple but technical and is therefore deferred to Appendix
I. For any measure $\mu$ on $\mathbb{R}$, let $\left.\mu\right|_{A}\left(B\right)=\mu\left(A\cap B\right)$
denote the restriction of $\mu$ to a measurable set $A$. 
\begin{lem}
\label{lem:apx.Ligg}Let $\eta=\sum_{i}\delta_{\eta_{i}}$ and $\eta_{N}=\sum_{i}\delta_{\eta_{N,i}}$,
$N\geq1$, be point processes such that $\eta_{N}\stackrel{d}{\to}\eta$.
Assume that the intensity function of $\eta$ is bounded from above
by $e^{ax}$, for some $a>0$. Let $X_{i}$ be a sequence of random
variables independent of $\eta$ and $\eta_{N}$, such that $\sup_i \mathbb P\{X_i\leq x\}\leq v(x)$, for some function $v(x)$ satisfying $\int v(-x) e^{ax}dx <\infty$.  Suppose that $\eta_{N}^{+}=\sum_{i}\delta_{\eta_{N,i}^{+}}$
is an additional sequence of point processes and, for any $L>0$ and
$N\geq1$, $\left(\bar{X}_{N,i}\left(L\right)\right)_{i\geq1}$ is
a sequence of random variables such that:
\begin{enumerate}
\item For any $\kappa>0$, assuming $\left.\eta_{N}\right|_{\left[-L,L\right]}=\sum_{i=1}^{Q_{N,L}}\delta_{\eta_{N,L,i}}$,
with $Q_{N,L}=\eta_{N}\left(\left[-L,L\right]\right)$, and defining
$\bar{\eta}_{N,L}\triangleq\sum_{i=1}^{Q_{N,L}}\delta_{\eta_{N,L,i}+\bar{X}_{N,i}\left(L\right)}$,
we have 
\begin{equation}
\lim_{L\to\infty}\liminf_{N\to\infty}\mathbb{P}\left\{ \left.\eta_{N}^{+}\right|_{\left[-\kappa,\kappa\right]}=\left.\bar{\eta}_{N,L}\right|_{\left[-\kappa,\kappa\right]}\right\} =1.\label{eq:15}
\end{equation}

\item \label{enu:2}For any fixed $m$ and $L$, $\left(\bar{X}_{N,i}\left(L\right)\right)_{i=1}^{m}$
converges in probability to $\left(X_{i}\right)_{i=1}^{m}$, as $N\to\infty$.
\end{enumerate}

Then, 
\[
\eta_{N}^{+}\stackrel[{\scriptstyle N\to\infty}]{d}{\to}\eta_{\infty}^{+}\triangleq\sum_{i}\delta_{\eta_{i}+X_{i}},
\]
and $\eta_{\infty}^{+}$ is locally finite.

\end{lem}
The following lemma establishes that, with high probability, for any
critical point of $H_{N}$ with critical value in a bounded set around
$m_{N}$ there is a critical point of $H_{N}^{+}$ within microscopic
distance on the sphere. Moreover, it expresses, up to a lower order
term, the difference of the two critical values in terms of the perturbation
field $H_{N}^{\prime}$ evaluated at the critical point of $H_{N}$.
We denote the set of critical points of $H_{N}$ with critical value
at distance $L$ at most from $m_{N}$ by 
\begin{equation}
\mathscr{C}_{N}\left(L\right)\triangleq\left\{ \boldsymbol{\sigma}\in\mathscr{C}_{N}:\,H_{N}\left(\boldsymbol{\sigma}\right)\in\left[m_{N}-L,\,m_{N}+L\right]\right\} ,\label{eq:25}
\end{equation}
and define $\mathscr{C}_{N}^{+}\left(L\right)$ similarly with $H_N^+(\boldsymbol{\sigma})$.
\begin{lem}
\label{lem:5}For any $\alpha<1/2$, 
there exists a sequence $\rho_{N}>0$ with $\lim_{N\to\infty}\rho_{N}=0$,
such that the following holds. For any fixed $L$, there exist a mapping
$\mathscr{G}_{N,L}:\,\mathscr{C}_{N}\left(L\right)\to\mathscr{C}_{N}^{+}$
and an event $\mathcal{E}_{N,L}^{1}$, such that on $\mathcal{E}_{N,L}^{1}$
for any $\boldsymbol{\sigma}\in\mathscr{C}_{N}\left(L\right)$, denoting
$\boldsymbol{\sigma}'=\mathscr{G}_{N,L}\left(\boldsymbol{\sigma}\right)$,
we have 
\begin{align}
&(1)\mbox{ Value perturbation:\,\,} \quad \left|H_{N}^{+}\left(\boldsymbol{\sigma}'\right)-\left(H_{N}\left(\boldsymbol{\sigma}\right)+\frac{1}{\sqrt{N}}H_{N}^{\prime}\left(\boldsymbol{\sigma}\right)-C_{0}\right)\right|\leq\rho_{N},\nonumber \\
&(2)\mbox{ Location perturbation:\,\,} \qquad R\left(\boldsymbol{\sigma},\boldsymbol{\sigma}'\right)\geq1-N^{-2\alpha},\label{eq:28}
\end{align}
and such that 
\[
\lim_{N\to\infty}\mathbb{P}\left\{ \mathcal{E}_{N,L}^{1}\right\} =1.
\]

\end{lem}
The proof of Lemma \ref{lem:5} is rather lengthy; Section \ref{sec:pfProp5}
is devoted to it. Note that no claim was made in Lemma \ref{lem:5}
that the mapping $\mathscr{G}_{N,L}$ is one-to-one. The following
lemma says that the near minimum critical points of $H_{N}$ are far
apart from each other which, combined with point (2) of Lemma \ref{lem:5},
allows us to conclude that indeed $\mathscr{G}_{N,L}$ is one-to-one,
with high probability. Another important consequence of the following lemma
is that the perturbations $H_{N}^{\prime}\left(\boldsymbol{\sigma}\right)/\sqrt{N}$
appearing in point (1) of Lemma \ref{lem:5} are asymptotically independent.
\begin{lem}
\label{lem:separation}There exists a sequence $r_{N}\in\left(0,1\right)$
with $r_{N}\to0$, such that for any $L>0$, 
\begin{equation}
\lim_{N\to\infty}\mathbb{P}\left\{ \forall\boldsymbol{\sigma}_{1}\neq\pm\boldsymbol{\sigma}_{2}\in\mathscr{C}_{N}\left(L\right):\,\left|R\left(\boldsymbol{\sigma}_{1},\boldsymbol{\sigma}_{2}\right)\right|\leq r_{N}\right\} =1,\label{eq:26}
\end{equation}
and similarly for $\mathscr{C}_{N}^{+}\left(L\right)$.
\end{lem}
The last ingredient we need is to show that critical points of $H_{N}^{+}$
are covered in some sense by the image of $\mathscr{G}_{N,L}$. This
is the content of the following lemma.
\begin{lem}
\label{lem:cover}For any fixed $\kappa>0$, with $\mathscr{G}_{N,L}$
as defined in Lemma \ref{lem:5}, 
\begin{equation}
\lim_{L\to\infty}\liminf_{N\to\infty}\mathbb{P}\left\{ \mathscr{C}_{N}^{+}\left(\kappa\right)\subset\mathscr{G}_{N,L}\left(\mathscr{C}_{N}\left(L\right)\right)\right\} =1.\label{eq:27}
\end{equation}

\end{lem}

Next, we explain how Theorem  \ref{thm:subseq} follows if we assume the lemmas above. In Sections \ref{subsec:pf1} and \ref{subsec:pf2} we prove Lemmas \ref{lem:separation} and \ref{lem:cover} (assuming Lemma \ref{lem:5}). As mentioned before, the proof of  Lemma \ref{lem:5} will be given in the following section.

In those proofs we shall need notation related to the above which we now introduce. Lemma \ref{lem:5} is a statement on the event $\mathcal{E}_{N,L}^{1}$ defined there. Similarly,  define the events $\mathcal{E}_{N,L}^{2}$ and $\mathcal{E}_{N,L}^{3}$ by
\begin{align*}
\mathcal{E}_{N,L}^{2}&=\left\{ \forall\boldsymbol{\sigma}_{1}\neq\pm\boldsymbol{\sigma}_{2}\in\mathscr{C}_{N}\left(L\right):\,\left|R\left(\boldsymbol{\sigma}_{1},\boldsymbol{\sigma}_{2}\right)\right|\leq r_{N}\right\},\\ 
\mathcal{E}_{N,L}^{3}&=\left\{ \mathscr{C}_{N}^{+}\left(\kappa\right)\subset\mathscr{G}_{N,L}\left(\mathscr{C}_{N}\left(L\right)\right)\right\} ,
\end{align*}
so that Lemmas \ref{lem:separation} and \ref{lem:cover} are statements on those events.
Also set 
\begin{align*}
	\mathcal{E}_{N,L}^\prime&=\mathcal{E}_{N,L}^{1}\cap\mathcal{E}_{N,L}^{2},\\
	\mathcal{E}_{N,L}&=\mathcal{E}_{N,L}^\prime\cap\mathcal{E}_{N,L}^{3}.
\end{align*}
Define
\begin{equation}
Q_{N,L}\triangleq\xi_{N}\left(\left[-L,\,L\right)\right)=\left(1+\iota_{p}\right)^{-1}\left|\mathscr{C}_{N}\left(L\right)\right|.%\label{eq:Q_NL}
\end{equation}
For odd $p$, let $\boldsymbol{\sigma}_{i}$, $i\leq Q_{N,L}$,
be an enumeration of $\mathscr{C}_{N}\left(L\right)$, and for even
$p$ let it be an enumeration of $\bar{\mathscr{C}}_{N}\left(L\right)=\left\{ \boldsymbol{\sigma}\in\mathscr{C}_{N}\left(L\right):\,\left\langle \boldsymbol{\sigma},\mathbf{n}\right\rangle >0\right\} $,
with $\mathbf{n}=\left(0,...,0,1\right)$.\footnote{This set almost surely contains exactly half of the points in $\mathscr{C}_{N}\left(L\right)$,
and by definition does not contain antipodal points.} For any $i\leq Q_{N,L}$ we define the following. First, we let $T_{N,i}\left(L\right)$ denote the error term  of point (1) of Lemma \ref{lem:5} that corresponds to $\boldsymbol{\sigma}_{i}$,
\[
T_{N,i}\left(L\right)  =H_{N}^{+}\left(\mathscr{G}_{N,L}\left(\boldsymbol{\sigma}_{i}\right)\right)-\left(H_{N}\left(\boldsymbol{\sigma}_{i}\right)+\frac{1}{\sqrt{N}}H_{N}^{\prime}\left(\boldsymbol{\sigma}_{i}\right)-C_{0}\right).
\]
Second, $\tilde{X}_{N,i}\left(L\right)$ denotes the expected shift of the critical value resulted from the perturbation of the Hamiltonian,
\[
\tilde{X}_{N,i}\left(L\right)  =\frac{1}{\sqrt{N}}H_{N}^{\prime}\left(\boldsymbol{\sigma}_{i}\right)-C_{0}.
\]
And third, $\bar{X}_{N,i}\left(L\right)$ is the corresponding actual shift,
\[
\bar{X}_{N,i}\left(L\right)  =\tilde{X}_{N,i}\left(L\right)+T_{N,i}\left(L\right)=H_{N}^{+}\left(\mathscr{G}_{N,L}\left(\boldsymbol{\sigma}_{i}\right)\right)-H_{N}\left(\boldsymbol{\sigma}_{i}\right).
\]
For any $i>Q_{N,L}$ set $\bar{X}_{N,i}\left(L\right)=W_{i}-C_{0}$
and $T_{N,i}\left(L\right)=0$, where $W_{i}\sim N\left(0,1\right)$
are independent of each other and any other variables.

\subsection{\label{sec:pfThem subseq}Proof of Theorem \ref{thm:subseq}, assuming
	Lemmas \ref{lem:apx.Ligg}, \ref{lem:5}, \ref{lem:separation} and \ref{lem:cover}}

By Lemmas \ref{lem:5}, \ref{lem:separation} and \ref{lem:cover},
\[
\lim_{L\to\infty}\liminf_{N\to\infty}\mathbb{P}\left\{ \mathcal{E}_{N,L}\right\} =1.
\]
On the event $\mathcal{E}_{N,L}$,
from point (2) of Lemma \ref{lem:5} and the definition of $\mathcal{E}_{N,L}^{2}$,
for large enough $N$, $\mathscr{G}_{N,L}$ is one-to-one. Combined
with the definition of $\mathcal{E}_{N,L}^{3}$ this yields, on $\mathcal{E}_{N,L}$,\footnote{Up to the negligible event that there exists $\boldsymbol{\sigma}\in\mathscr{C}_{N}\left(L\right)$
	with $\left\langle \boldsymbol{\sigma},\mathbf{n}\right\rangle =0$.}
\begin{equation}
\left.\xi_{N}^{+}\right|_{\left[-\kappa,\kappa\right]}=\left.\left(\sum_{i=1}^{Q_{N,L}}\delta_{H_{N}\left(\boldsymbol{\sigma}_{i}\right)-m_{N}+\bar{X}_{N,i}\left(L\right)}\right)\right|_{\left[-\kappa,\kappa\right]}.\label{eq:30}
\end{equation}

Now, let $\xi_{N_{k}}$ be an arbitrary convergent subsequence of
$\xi_{N}$ with limit in distribution $\bar{\xi}_{\infty}=\sum\delta_{\bar{\xi}_{\infty,i}}$.
We wish to show that the corresponding subsequence $\xi_{N_{k}}^{+}$ of
$\xi_{N}^{+}$ converges to $\sum\delta_{\bar{\xi}_{\infty,i}+W_{i}-C_{0}}$
in distribution. For odd $p$, since $H_{N}\left(-\boldsymbol{\sigma}\right)=-H_{N}\left(\boldsymbol{\sigma}\right)$,
if $\boldsymbol{\sigma}\in\mathscr{C}_{N}\left(L\right)$, then $-\boldsymbol{\sigma}\notin\mathscr{C}_{N}\left(L\right)$,
for large enough $N$. For even $p$, the same follows for $\bar{\mathscr{C}}_{N}\left(L\right)$,
by its definition. Note that $H_{N}^{\prime}$ is independent of $\boldsymbol{\sigma}_{i}$,
$i\leq Q_{N,L}$, and recall the bound we have on $T_{N,i}\left(L\right)$
from point (1) of Lemma \ref{lem:5}. Combined with Lemma \ref{lem:separation}
those imply that for any finite $m$, $\left(\bar{X}_{N,i}\left(L\right)\right)_{i=1}^{m}$
converges in probability to $\left(W_{i}-C_{0}\right)_{i=1}^{m}$
as $N\to\infty$, where $W_{i}\sim N\left(0,1\right)$ are i.i.d and
independent of $\bar{\xi}_{\infty}$. From Proposition \ref{prop:intensity} the density function of $\bar \xi_\infty$ is bounded by $e^{c_px}$. We showed that the conditions
of Lemma \ref{lem:apx.Ligg} are met with $\eta$, $\eta_N$, $\eta_N^+$ and $X_i$ in the statement of the lemma corresponding to $\bar \xi_\infty$, $\xi_{N_k}$, $\xi_{N_k}^+$ and $W_i-C_0$ in the current setting.
It follows that $\xi_{N_{k}}^{+}$
converges in distribution to $\sum\delta_{\bar{\xi}_{\infty,i}+W_{i}-C_{0}}$.
This completes the proof. \qed

\subsection{\label{subsec:pf1}Proof of Lemma \ref{lem:separation}}

Let $\delta$, $\epsilon>0$. With the notation we used in the proof of Proposition \ref{prop:2nd moment}, for large enough $N$,
\begin{equation}
\mathbb{E}\left\{ \forall\boldsymbol{\sigma}_{1}\neq\pm\boldsymbol{\sigma}_{2}\in\mathscr{C}_{N}\left(L\right):\,\left|R\left(\boldsymbol{\sigma}_{1},\boldsymbol{\sigma}_{2}\right)\right|\geq \epsilon \right\} \leq
\mathbb{E}\left\{ \left[\mbox{Crt}_{N}\left(B_{\delta},I_{\epsilon}\right)\right]_{2}\right\},\label{eq:l4}
\end{equation}
where  $B_{\delta}=\left(-E_{0}-\delta,-E_{0}+\delta\right)$
and $I_{\epsilon}=\left(-1,1\right)\setminus\left(-\epsilon,\epsilon\right)$.
In the proof of Proposition \ref{prop:2nd moment} we showed that, with $\epsilon$ fixed and small enough $\delta$,
\begin{equation*}
\lim_{N\to\infty}\frac{1}{N}\log\left(\mathbb{E}\left\{ \left[\mbox{Crt}_{N}\left(B_{\delta},I_{\epsilon}\right)\right]_{2}\right\} \right) < 0.
\end{equation*}
In particular, as $N\to\infty$, the left-hand side of \eqref{eq:l4} goes to $0$. The lemma follows from this and Markov's inequality (where we note that the statement about $\mathscr{C}_{N}^+\left(L\right)$ from the fact that $H_N^+(\boldsymbol{\sigma})$ has the same law as $\sqrt{(N+1)/N}H_N(\boldsymbol{\sigma})$).
 \qed

\subsection{\label{subsec:pf2}Proof of Lemma \ref{lem:cover} (assuming Lemma \ref{lem:5})}

By the argument that led to (\ref{eq:29-1}) and by Proposition \ref{prop:intensity},
as $N\to\infty$, 
\begin{align}
\lim_{N\to\infty}\mathbb{E}\left\{ \xi_{N}^{+}\left(\left[-\kappa,\,\kappa\right]\right)\right\} \nonumber 
 & =\lim_{N\to\infty}\mathbb{E}\left\{ \xi_{N}\left(\left[\frac{1}{2}E_{0}-\kappa,\,\frac{1}{2}E_{0}+\kappa\right]\right)\right\} \nonumber \\
 & =\int_{\frac{1}{2}E_{0}-\kappa}^{\frac{1}{2}E_{0}+\kappa}e^{c_{p}x}dx\nonumber \\
 & =\frac{1}{c_{p}}e^{\frac{1}{2}c_{p}E_{0}}\left(e^{c_{p}\kappa}-e^{-c_{p}\kappa}\right).\label{eq:43}
\end{align}

On $\mathcal{E}_{N,L}^{\prime}$ for large enough $N$, $\mathscr{G}_{N,L}$
is one-to-one and its image is contained in $\mathscr{C}_{N}^{+}$.
Thus,
\begin{equation}
\left(\mathbf{1}_{\mathcal{E}_{N,L}^{\prime}}\sum_{i=1}^{Q_{N,L}}\delta_{H_{N}^+\left(\mathscr{G}_{N,L}\left(\boldsymbol{\sigma}_{i}\right)\right)-m_{N}}\right)\left(\left[-\kappa,\kappa\right]\right)\leq\xi_{N}^{+}\left(\left[-\kappa,\kappa\right]\right).\label{eq:a1}
\end{equation}
Hence, by applying Markov's inequality to the difference of the two
sides of (\ref{eq:a1}), in order to prove the lemma, it will be sufficient
to show that
\begin{equation}
\liminf_{L\to\infty}\liminf_{N\to\infty}\mathbb{E}\left\{ \left(\mathbf{1}_{\mathcal{E}_{N,L}^{\prime}}\sum_{i=1}^{Q_{N,L}}\delta_{H_{N}^+\left(\mathscr{G}_{N,L}\left(\boldsymbol{\sigma}_{i}\right)\right)-m_{N}}\right)\left(\left[-\kappa,\kappa\right]\right)\right\} \label{eq:31}
\end{equation}
is greater than or equal to (\ref{eq:43}) (and therefore equal). 

From the bound we have on $T_{N,i}\left(L\right)$ in point (1) of
Lemma \ref{lem:5}, for large
enough $N$,
\begin{align}
 & \!\!\!\! \mathbb{E}\left\{ \left(\mathbf{1}_{\mathcal{E}_{N,L}^{\prime}}\sum_{i=1}^{Q_{N,L}}\delta_{H_{N}^+\left(\mathscr{G}_{N,L}\left(\boldsymbol{\sigma}_{i}\right)\right)-m_{N}}\right)\left(\left[-\kappa,\kappa\right]\right)\right\}\nonumber \\ 
 &  =\mathbb{E}\left\{ \left(\mathbf{1}_{\mathcal{E}_{N,L}^{\prime}}\sum_{i=1}^{Q_{N,L}}\delta_{H_{N}\left(\boldsymbol{\sigma}_{i}\right)+\frac{1}{\sqrt{N}}H_{N}^\prime\left(\boldsymbol{\sigma}_{i}\right)-C_0-m_{N}+T_{N,i}(L)}\right)\left(\left[-\kappa,\kappa\right]\right)\right\}\nonumber \\ 
 &  \geq \mathbb{E}\left\{ \left(\mathbf{1}_{\mathcal{E}_{N,L}^{\prime}}\sum_{i=1}^{Q_{N,L}}\delta_{H_{N}\left(\boldsymbol{\sigma}_{i}\right)+\frac{1}{\sqrt{N}}H_{N}^\prime\left(\boldsymbol{\sigma}_{i}\right)-C_0-m_{N}}\right)\left(\left[-\kappa+\rho_N,\kappa-\rho_N\right]\right)\right\}\label{eq:l6} \\ 
 &  \geq \mathbb{E}\left\{ \left(\sum_{i=1}^{Q_{N,L}}\delta_{H_{N}\left(\boldsymbol{\sigma}_{i}\right)+\frac{1}{\sqrt{N}}H_{N}^\prime\left(\boldsymbol{\sigma}_{i}\right)-C_0-m_{N}}\right)\left(\left[-\kappa+\rho_N,\kappa-\rho_N\right]\right)\right\}\nonumber \\ 
 &-\mathbb{E}\left\{ \mathbf{1}_{\left(\mathcal{E}_{N,L}^{\prime}\right)^c} Q_{N,L}\right\}, \nonumber
\end{align}
where $\rho_{N}$ is defined in Lemma \ref{lem:5}.

From the Cauchy-Schwarz inequality, Propositions \ref{prop:intensity}
and \ref{prop:2nd moment} and Lemmas \ref{lem:5} and \ref{lem:separation},
\[
\lim_{L\to\infty}\limsup_{N\to\infty}\mathbb{E}\left\{ \mathbf{1}_{\left(\mathcal{E}_{N,L}^{\prime}\right)^c} Q_{N,L}\right\}\leq
\lim_{L\to\infty}\limsup_{N\to\infty}\left(\mathbb{E}\left\{ \left(Q_{N,L}\right)^{2}\right\} \mathbb{P}\left\{ \left(\mathcal{E}_{N,L}^{\prime}\right)^{c}\right\} \right)^{1/2}=0,
\]
since the limsup in $N$ of the second moment above is finite and the limit in $N$ of the probability goes to $0$.
Thus, when taking limits in \eqref{eq:l6} as in \eqref{eq:31}, the term corresponding to $\mathbb{E}\left\{ \mathbf{1}_{\left(\mathcal{E}_{N,L}^{\prime}\right)^c} Q_{N,L}\right\}$ vanishes.

Since $H_{N}^{\prime}\left(\boldsymbol{\sigma}_{i}\right)/\sqrt{N}\sim N\left(0,1\right)$
and the field $\left\{ H_{N}^{\prime}\left(\boldsymbol{\sigma}_{i}\right) \right\}$ is independent of the field $\left\{ H_{N}\left(\boldsymbol{\sigma}_{i}\right) \right\}$,
\begin{align*}
&\mathbb{E}\left\{ \left(\sum_{i=1}^{Q_{N,L}}\delta_{H_{N}\left(\boldsymbol{\sigma}_{i}\right)+\frac{1}{\sqrt{N}}H_{N}^\prime\left(\boldsymbol{\sigma}_{i}\right)-C_0-m_{N}}\right)\left(\left[-\kappa+\rho_N,\kappa-\rho_N\right]\right)\right\}\\
&\quad =\mu_{N,L} \star\mathbb{P}\left\{ W-C_{0}\in\,\cdot\,\right\} \left(\left[-\kappa+\rho_{N},\kappa-\rho_{N}\right]\right),
\end{align*}
where  $W\sim N\left(0,1\right)$ and $\mu_{N,L}$ is the intensity measure of $\sum_{i=1}^{Q_{N,L}}\delta_{H_{N}\left(\boldsymbol{\sigma}_{i}\right)-m_{N}}=\left.\xi_{N}\right|_{\left[-L,L\right]}$.

Therefore, by Proposition \ref{prop:intensity}, the limit in (\ref{eq:31})
is greater than or equal to 
\begin{align}
 & \lim_{L\to\infty}\lim_{N\to\infty}\mu_{N,L}\star\mathbb{P}\left\{ W-C_{0}\in\,\cdot\,\right\} \left(\left[-\kappa+\rho_{N},\kappa-\rho_{N}\right]\right)\nonumber \\
 & =\mu_{c_{p}}\star\mathbb{P}\left\{ W-C_{0}\in\,\cdot\,\right\} \left(\left[-\kappa,\kappa\right]\right)\nonumber \\
 & =\frac{1}{c_{p}}\exp\left\{ C_{0}c_{p}+\frac{1}{2}c_{p}^{2}\right\} \left(e^{c_{p}\kappa}-e^{-c_{p}\kappa}\right),\label{eq:32}
\end{align}
where $\mu_{c_{p}}\left(A\right)=\int_{A}e^{c_{p}x}dx$ for measurable
$A$. From (\ref{eq:C0E0rel}) it follows that (\ref{eq:43}) and
(\ref{eq:32}) are equal, which completes the proof.\qed

\section{\label{sec:pfProp5}Proof of Lemma \ref{lem:5}}

We fix throughout the section $\alpha\in\left(1/3,1/2\right)$, $\epsilon\in\left(0,1/2-\alpha\right)$,
and $\delta\in\left(0,p\left(E_{0}-E_{\infty}\right)\right)$ and
note that Lemma \ref{lem:5}  follows from this with general $\alpha<1/2$. Below we define a notion of $\left(\alpha,\epsilon,\delta\right)$-good
critical points. The lemma follows from the following two.
\begin{lem}
\label{lem:good_suff}A function $\mathscr{G}_{N,L}:\,\mathscr{C}_{N}\left(L\right)\to\mathscr{C}_{N}^{+}$
can be defined on the event that all $\boldsymbol{\sigma}\in\mathscr{C}_{N}\left(L\right)$
are $\left(\alpha,\epsilon,\delta\right)$-good such that, denoting
$\boldsymbol{\sigma}'=\mathscr{G}_{N,L}\left(\boldsymbol{\sigma}\right)$,
(\ref{eq:28}) holds.
\end{lem}

\begin{lem}
\label{lem:good_whp}The event that all $\boldsymbol{\sigma}\in\mathscr{C}_{N}\left(L\right)$
are $\left(\alpha,\epsilon,\delta\right)$-good has probability approaching
$1$ as $N\to\infty$.
\end{lem}

In Section \ref{sec:apx} we introduce the linear and quadratic approximations of the Hamiltonian $H_{N}\left(\boldsymbol{\sigma}\right)$ and the perturbed Hamiltonian $H_{N}^{+}\left(\boldsymbol{\sigma}\right)$ and related error functions. In Section \ref{sub:Approximate-critical-points} corresponding critical points and values will be discussed. Those will be used in Section \ref{sub:Good-critical-points} to define good critical points through bounds on various quantities. The proof of Lemma \ref{lem:good_suff} will then be given in Section \ref{sub:pf_good_suff}. 

In Sections \ref{sub:Auxiliary-results} and \ref{sub:pf_good_whp} we state several lemmas and explain 
how Lemma \ref{lem:good_whp} follows from them. 
Sections \ref{subsec:covs} and \ref{sub:Metric-entropies} are dedicated to results concerned with 
the covariance structure of the Hamiltonian, its gradient and its Hessian, and the metric entropy of the 
error functions related to the linear and quadratic approximations of Section \ref{sec:apx}.
The latter are used in Section \ref{sub:lemmas}, where the quoted
lemmas from Section \ref{sub:Auxiliary-results} are proved and the proof of Lemma \ref{lem:good_whp} is completed.

\subsection{\label{sec:apx}Linear and quadratic approximations}

It will be convenient to work on the Euclidean space instead of the
sphere directly, therefore we define the following. For any $x=\left(x_{1},...,x_{N-1}\right)\in\mathbb{R}^{N-1}$
with $\left\Vert x\right\Vert \leq1$ and any $y\in\mathbb{R}^{N}$
let 
\begin{align*}
	P\left(x\right) & =\left(x_{1},...,x_{N-1},\sqrt{1-\left\Vert x\right\Vert ^{2}}\right),\\
	S\left(y\right) & =\sqrt{N}y.
\end{align*}
For any $\boldsymbol{\sigma}\in\mathbb{S}^{N-1}=\left\{ \boldsymbol{\sigma}\in\mathbb{R}^{N}:\,\left\Vert \boldsymbol{\sigma}\right\Vert _{2}=1\right\} $
let $\theta_{\boldsymbol{\sigma}}$ be a rotation such that, with $\mathbf{n}\triangleq\left(0,...,0,1\right)$, 
\[
\theta_{\boldsymbol{\sigma}}\left(\mathbf{n}\right)=\boldsymbol{\sigma},
\]
and define 
\begin{equation}
	\bar{f}_{\boldsymbol{\sigma}}\left(x\right)\triangleq\frac{1}{\sqrt{N}}H_{N}\circ\theta_{\boldsymbol{\sigma}}\circ S\circ P\left(x\right),\label{eq:33}
\end{equation}
so that $\bar{f}_{\boldsymbol{\sigma}}\left(x\right)$ is a reparametrization
of the restriction of $H_{N}$ to the hemisphere around $\boldsymbol{\sigma}$,
normalized to have constant variance $1$. For $\boldsymbol{\sigma}=\mathbf{n}$, which due to stationarity we will be able to relate to the general case easily, we shall assume that the rotation $\theta_{\mathbf{n}}$ is the identity map. Define $\bar{f}_{\boldsymbol{\sigma}}^{\prime}\left(x\right)$
and $\bar{f}_{\boldsymbol{\sigma}}^{+}\left(x\right)$ similarly,
with $H_{N}$ replaced by $H_{N}^{\prime}$ and $H_{N}^{+}$, respectively.
The covariance function of $\bar{f}_{\boldsymbol{\sigma}}\left(x\right)$
is given by
\begin{align}
	\mathbb{E}\left\{ \bar{f}_{\boldsymbol{\sigma}}\left(x\right)\bar{f}_{\boldsymbol{\sigma}}\left(y\right)\right\}  & =\left(\left\langle x,y\right\rangle +\sqrt{1-\left\Vert x\right\Vert ^{2}}\sqrt{1-\left\Vert y\right\Vert ^{2}}\right)^{p}\triangleq\left(W\left(x,y\right)\right)^{p}.\label{eq:W}
\end{align}

Define the linear and quadratic approximations
\begin{align}
	\bar{f}_{\boldsymbol{\sigma},lin}^{\prime}\left(x\right) & =\bar{f}_{\boldsymbol{\sigma}}^{\prime}\left(0\right)+\left\langle \nabla\bar{f}_{\boldsymbol{\sigma}}^{\prime}\left(0\right),x\right\rangle ,\label{eq:flin}\\
	\bar{f}_{\boldsymbol{\sigma},quad}\left(x\right) & =\bar{f}_{\boldsymbol{\sigma}}\left(0\right)+\left\langle \nabla\bar{f}_{\boldsymbol{\sigma}}\left(0\right),x\right\rangle +\frac{1}{2}x^{T}\nabla^{2}\bar{f}_{\boldsymbol{\sigma}}\left(0\right)x,\label{eq:fquad}
\end{align}
where $\nabla$ and $\nabla^{2}$ are the usual Euclidean gradient
and Hessian. Define, for any $x\in\mathbb{R}^{N-1}$,
\begin{equation}
	\bar{f}_{\boldsymbol{\sigma}}^{\left(1\right)}\left(x\right)=\bar{f}_{\boldsymbol{\sigma}}^{\prime}\left(x\right)-\bar{f}_{\boldsymbol{\sigma},lin}^{\prime}\left(x\right)\,\,\mbox{and}\,\,\bar{f}_{\boldsymbol{\sigma}}^{\left(2\right)}\left(x\right)=\bar{f}_{\boldsymbol{\sigma}}\left(x\right)-\bar{f}_{\boldsymbol{\sigma},quad}\left(x\right).\label{eq:f(i)}
\end{equation}
From (\ref{eq:H+}),
\begin{equation}
	\sqrt{N}\bar{f}_{\boldsymbol{\sigma}}^{+}\left(x\right)=\sqrt{N}\bar{f}_{\boldsymbol{\sigma}}\left(x\right)+\bar{f}_{\boldsymbol{\sigma}}^{\prime}\left(x\right).\label{eq:3}
\end{equation}

In Section \ref{sub:Approximate-critical-points} we investigate
how the perturbation $\bar{f}_{\boldsymbol{\sigma}}^{\prime}\left(x\right)$
affect the critical value of $\sqrt{N}\bar{f}_{\boldsymbol{\sigma}}^{+}\left(x\right)$
when we assume that $x=0$ is a critical point of $\bar{f}_{\boldsymbol{\sigma}}\left(x\right)$
and approximate $\bar{f}_{\boldsymbol{\sigma}}^{\prime}\left(x\right)$
and $\bar{f}_{\boldsymbol{\sigma}}\left(x\right)$ by (\ref{eq:flin})
and (\ref{eq:fquad}), respectively. Section \ref{sub:Metric-entropies}
deals with metric entropy bounds related to the fields $\bar{f}_{\boldsymbol{\sigma}}^{\left(1\right)}\left(x\right)$
and $\bar{f}_{\boldsymbol{\sigma}}^{\left(2\right)}\left(x\right)$, which will be used to control the (maximum of the) error functions \eqref{eq:f(i)}.

\subsection{\label{sub:Approximate-critical-points}Approximate critical points
	and values of $\sqrt{N}\bar{f}_{\boldsymbol{\sigma}}^{+}\left(x\right)$}

Suppose $x=0$ is a critical point of $\sqrt{N}\bar{f}_{\boldsymbol{\sigma}}\left(x\right)$,
i.e. $\nabla\bar{f}_{\boldsymbol{\sigma}}\left(0\right)=0$. Define
the function 
\begin{align}
	\sqrt{N}\bar{f}_{\boldsymbol{\sigma},apx}^{+}\left(x\right) & \triangleq\sqrt{N}\bar{f}_{\boldsymbol{\sigma}}^{+}\left(x\right)-\sqrt{N}\bar{f}_{\boldsymbol{\sigma}}^{\left(2\right)}\left(x\right)-\bar{f}_{\boldsymbol{\sigma}}^{\left(1\right)}\left(x\right)\label{eq:5}\\
	& =\sqrt{N}\bar{f}_{\boldsymbol{\sigma},quad}\left(x\right)+\bar{f}_{\boldsymbol{\sigma},lin}^{\prime}\left(x\right).\nonumber 
\end{align}
Whenever the Hessian $\nabla^{2}\bar{f}_{\boldsymbol{\sigma}}\left(0\right)$
is invertible, and therefore $\sqrt{N}\bar{f}_{\boldsymbol{\sigma},apx}^{+}\left(x\right)$
is non-degenerate as a quadratic function, define 
\[
Y_{\boldsymbol{\sigma}}:=Y_{\boldsymbol{\sigma},N}\in\mathbb{R}^{N-1}\mbox{\,\,\ and\,\,}V_{\boldsymbol{\sigma}}:=V_{\boldsymbol{\sigma},N}\in\mathbb{R},
\]
to be the critical point and critical value of $\sqrt{N}\bar{f}_{\boldsymbol{\sigma},apx}^{+}\left(x\right)$,
respectively. If the Hessian is singular, set them arbitrarily to
be equal to $0$ (of course, this value will not affect our analysis
in the sequel). 

We now express $Y_{\boldsymbol{\sigma}}$ and $V_{\boldsymbol{\sigma}}$,
assuming $\nabla^{2}\bar{f}_{\boldsymbol{\sigma}}\left(0\right)$
is invertible, as functions of $\bar{f}_{\boldsymbol{\sigma}}^{\prime}\left(0\right)$,
$\nabla\bar{f}_{\boldsymbol{\sigma}}^{\prime}\left(0\right)$, and
$\nabla^{2}\bar{f}_{\boldsymbol{\sigma}}\left(0\right)$. By differentiation,
\[
\sqrt{N}\nabla^{2}\bar{f}_{\boldsymbol{\sigma}}\left(0\right)Y_{\boldsymbol{\sigma}}+\nabla\bar{f}_{\boldsymbol{\sigma}}^{\prime}\left(0\right)=0,
\]
and thus
\begin{equation}
	Y_{\boldsymbol{\sigma}}=-\frac{1}{\sqrt{N}}\left(\nabla^{2}\bar{f}_{\boldsymbol{\sigma}}\left(0\right)\right)^{-1}\nabla\bar{f}_{\boldsymbol{\sigma}}^{\prime}\left(0\right).\label{eq:4}
\end{equation}

By substitution in (\ref{eq:5}), $V_{\boldsymbol{\sigma}}=\sqrt{N}\bar{f}_{\boldsymbol{\sigma}}\left(0\right)+\bar{f}_{\boldsymbol{\sigma}}^{\prime}\left(0\right)+\Delta_{\boldsymbol{\sigma}}$
with 
\begin{align}
	\Delta_{\boldsymbol{\sigma}} & =\frac{\sqrt{N}}{2}Y_{\boldsymbol{\sigma}}^{T}\nabla^{2}\bar{f}_{\boldsymbol{\sigma}}\left(0\right)Y_{\boldsymbol{\sigma}}+\left\langle \nabla\bar{f}_{\boldsymbol{\sigma}}^{\prime}\left(0\right),Y_{\boldsymbol{\sigma}}\right\rangle \nonumber \\
	& =-\frac{1}{2\sqrt{N}}\left(\nabla\bar{f}_{\boldsymbol{\sigma}}^{\prime}\left(0\right)\right)^{T}\left(\nabla^{2}\bar{f}_{\boldsymbol{\sigma}}\left(0\right)\right)^{-1}\nabla\bar{f}_{\boldsymbol{\sigma}}^{\prime}\left(0\right)\label{eq:10}
\end{align}
(where here too $\Delta_{\boldsymbol{\sigma}}$ can be taken to be
$0$ on the event that $\nabla^{2}\bar{f}_{\boldsymbol{\sigma}}\left(0\right)$
is singular). In the sequel, we shall treat $Y_{\boldsymbol{\sigma}}$
and $\Delta_{\boldsymbol{\sigma}}$ as random fields on $\mathbb{S}^{N-1}$.

\subsection{\label{sub:Good-critical-points}Good critical points}

Below we define the functions $g_{i}=g_{i,N}$ on the sphere $\mathbb{S}^{N-1}$
and the corresponding sets $B_{i}=B_{i,N}$, $1\leq i\leq8$. Write 
\[
\mathcal{A}_{i}\left(\boldsymbol{\sigma}\right)=\mathcal{A}_{i,N}\left(\boldsymbol{\sigma}\right)\triangleq\left\{ g_{i,N}\left(\boldsymbol{\sigma}\right)\in B_{i,N}\right\} .
\]
With $\boldsymbol{\sigma}\in\mathbb{S}^{N-1}$, we say
that a critical point $\sqrt{N}\boldsymbol{\sigma}$ of the Hamiltonian
$H_{N}$\emph{ is $\left(\alpha,\delta,\epsilon\right)$-good }(or
simply \emph{good}) if $\mathcal{A}_{i}(\boldsymbol{\sigma})$ occurs for all $1\leq i\leq8$
(where we recall that in the beginning of Section \ref{sec:pfProp5} we assumed that $\alpha\in\left(1/3,1/2\right)$, $\epsilon\in\left(0,1/2-\alpha\right)$,
and $\delta\in\left(0,p\left(E_{0}-E_{\infty}\right)\right)$).

Recall that, with 
\[
\mathcal{B}_{N-1}\left(0,R\right)\triangleq\left\{ x\in\mathbb{R}^{N-1}:\,\left\Vert x\right\Vert \leq R\right\} ,
\]
denoting the Euclidean ball, for $\boldsymbol{\sigma}\in\mathbb{S}^{N-1}$, $\bar{f}_{\boldsymbol{\sigma}}:\mathcal{B}_{N-1}\left(0,1\right)\to\mathbb{R}$,
defined in (\ref{eq:33}), is the reparametrization of the restriction
of $\frac{1}{\sqrt{N}}H_{N}$ to the hemisphere around $\sqrt{N}\boldsymbol{\sigma}$
obtained by rotating, projecting and scaling it.  Let $\mathcal{V}\left(\delta\right)$ be the
set of real, symmetric matrices with eigenvalues in the interval
\[
\left(pE_{0}-\left(2\sqrt{p\left(p-1\right)}+\delta\right),\,pE_{0}+2\sqrt{p\left(p-1\right)}+\delta\right).
\]
The first function and corresponding set we define are
\[
g_{1}\left(\boldsymbol{\sigma}\right)=\nabla^{2}\bar{f}_{\boldsymbol{\sigma}}\left(0\right),\qquad B_{1}=\sqrt{N}\mathcal{V}\left(\delta\right).
\]

The random fields $\bar{f}_{\boldsymbol{\sigma}}^{+}$ and $\bar{f}_{\boldsymbol{\sigma}}^{\prime}$
were also defined in Section \ref{sec:apx}, similarly to $\bar{f}_{\boldsymbol{\sigma}}$
only with $H_{N}^{+}$ and $H_{N}^{\prime}$, respectively. In our
study of the perturbations of critical points and values in the same
section we defined $Y_{\boldsymbol{\sigma}}$ (see (\ref{eq:4}))
as the critical point of $\sqrt{N}\bar{f}_{\boldsymbol{\sigma},apx}^{+}\left(x\right)$
defined in (\ref{eq:5})\footnote{At least when $\nabla^{2}\bar{f}_{\boldsymbol{\sigma}}\left(0\right)$
is invertible, which is the case when $\mathcal{A}_{1}$ occurs, and
we shall indeed restrict to this event when discussing $Y_{\boldsymbol{\sigma}}$
below.}. The latter was defined as an approximation to $\sqrt{N}\bar{f}_{\boldsymbol{\sigma}}^{+}\left(x\right)$
obtained by replacing $\bar{f}_{\boldsymbol{\sigma}}$ and $\bar{f}_{\boldsymbol{\sigma}}^{\prime}$
by their quadratic and linear approximations, respectively. We define
\[
g_{2}\left(\boldsymbol{\sigma}\right)=Y_{\boldsymbol{\sigma}},\qquad B_{2}=\mathcal{B}_{N-1}\left(0,N^{-\alpha}\right),
\]
which is related to the proof of point (2) of (\ref{eq:28}).

Together with $Y_{\boldsymbol{\sigma}}$ we defined the corresponding
approximated critical value by $V_{\boldsymbol{\sigma}}=\sqrt{N}\bar{f}_{\boldsymbol{\sigma}}\left(0\right)+\bar{f}_{\boldsymbol{\sigma}}^{\prime}\left(0\right)+\Delta_{\boldsymbol{\sigma}}$
where $\Delta_{\boldsymbol{\sigma}}$ is given in (\ref{eq:10}).
We define 
\begin{align*}
g_{3}\left(\boldsymbol{\sigma}\right) & =\Delta_{\boldsymbol{\sigma}}+\frac{p}{2\sqrt{N}}{\rm Tr}\left(\left(\nabla^{2}\bar{f}_{\boldsymbol{\sigma}}\left(0\right)\right)^{-1}\right), & B_{3}=\left(-N^{-\frac{1}{2}+\epsilon},N^{-\frac{1}{2}+\epsilon}\right),\\
g_{4}\left(\boldsymbol{\sigma}\right) & =\frac{p}{2\sqrt{N}}{\rm Tr}\left(\left(\nabla^{2}\bar{f}_{\boldsymbol{\sigma}}\left(0\right)\right)^{-1}\right), & B_{4}=\left(C_{0}-\tau_{\epsilon,\delta}\left(N\right),C_{0}+\tau_{\epsilon,\delta}\left(N\right)\right),
\end{align*}
where $C_{0}$ is defined in (\ref{eq:C0}) and $\tau_{\epsilon,\delta}\left(N\right)$
is a sequence of numbers such that $\lim_{N\to\infty}\tau_{\epsilon,\delta}\left(N\right)=0$ which will be assumed to be large enough whenever needed.
Since $\Delta_{\boldsymbol{\sigma}}=g_{3}\left(\boldsymbol{\sigma}\right)-g_{4}\left(\boldsymbol{\sigma}\right)$,
knowing that $\mathcal{A}_{3}(\boldsymbol{\sigma})$ and $\mathcal{A}_{4}(\boldsymbol{\sigma})$
occur allows us to conclude that $\Delta_{\boldsymbol{\sigma}}$ is close to $-C_0$.

Lastly, we recall the linear and quadratic approximations $\bar{f}_{\boldsymbol{\sigma},lin}^{\prime}\left(x\right)$
and $\sqrt{N}\bar{f}_{\boldsymbol{\sigma},quad}\left(x\right)$ defined
in (\ref{eq:flin}) and (\ref{eq:fquad}), and the corresponding `error'
functions $\bar{f}_{\boldsymbol{\sigma}}^{\left(1\right)}\left(x\right)$
and $\bar{f}_{\boldsymbol{\sigma}}^{\left(2\right)}\left(x\right)$
defined in (\ref{eq:f(i)}). With 
\begin{equation}
K_{p,\delta}=p\left(E_{0}-E_{\infty}\right)-\delta>0,\label{eq:Kp}
\end{equation}
we set
\begin{align*}
g_{5}\left(\boldsymbol{\sigma}\right) & =\sup_{x\in\mathcal{B}_{N-1}\left(0,N^{-\alpha}\right)}\left|\bar{f}_{\boldsymbol{\sigma}}^{\left(1\right)}\left(x\right)\right|, & B_{5}=\left(-CN^{\frac{1}{2}-2\alpha},CN^{\frac{1}{2}-2\alpha}\right),\\
g_{6}\left(\boldsymbol{\sigma}\right) & =\sup_{x\in\mathcal{B}_{N-1}\left(0,N^{-\alpha}\right)}\left|\sqrt{N}\bar{f}_{\boldsymbol{\sigma}}^{\left(2\right)}\left(x\right)\right|, & B_{6}=\left(-CN^{1-3\alpha},CN^{1-3\alpha}\right),\\
g_{7}\left(\boldsymbol{\sigma}\right) & =\inf_{x:\,\left\Vert x\right\Vert =N^{-\alpha}}\sqrt{N}\bar{f}_{\boldsymbol{\sigma},quad}\left(x\right)-\sqrt{N}\bar{f}_{\boldsymbol{\sigma}}\left(0\right), & B_{7}=\left(\frac{1}{2}K_{p,\delta}N^{1-2\alpha},\,\infty\right),\\
g_{8}\left(\boldsymbol{\sigma}\right) & =\inf_{x:\,\left\Vert x\right\Vert =N^{-\alpha}}\bar{f}_{\boldsymbol{\sigma},lin}^{\prime}\left(x\right)-\bar{f}_{\boldsymbol{\sigma}}^{\prime}\left(0\right), & B_{8}=\left(-N^{-\frac{1}{2}-\alpha+\epsilon},\,\infty\right),
\end{align*}
where $C>0$ is a constant which will be assumed to be large enough
whenever needed. The events $\mathcal{A}_{i}$, $5\leq i\leq8$, will
be used when we will need to show that the minimum of $\sqrt{N}\bar{f}_{\boldsymbol{\sigma}}^{+}$
in $\mathcal{B}_{N-1}\left(0,N^{-\alpha}\right)$ is attained at an
interior point (by comparing the minimum on the boundary to a particular
value attained at a point in the interior).

\subsection{\label{sub:pf_good_suff}Proof of Lemma \ref{lem:good_suff}}

Let $\sqrt{N}\boldsymbol{\sigma}\in\mathscr{C}_{N}$, which is equivalent
to $x=0$ being a critical point of $\sqrt{N}\bar{f}_{\boldsymbol{\sigma}}\left(x\right)$,
and suppose $\sqrt{N}\boldsymbol{\sigma}$ is a good critical point.
Use $(\mathcal A_i)$ as a shorthand to ``$\mathcal A_i(\boldsymbol{\sigma})$ occurs''.
Since $\left(\mathcal{A}_{1}\right)$, $\nabla^{2}\bar{f}_{\boldsymbol{\sigma}}\left(0\right)$
is invertible and $Y_{\boldsymbol{\sigma}}$ is the global minimum
point of the convex function $\sqrt{N}\bar{f}_{\boldsymbol{\sigma},apx}^{+}\left(x\right)$
and it is defined by (\ref{eq:4}). The corresponding minimal value
is, as $N\to\infty$, 
\begin{align}
\sqrt{N}\bar{f}_{\boldsymbol{\sigma},apx}^{+}\left(Y_{\boldsymbol{\sigma}}\right) & \overset{\eqref{eq:10}}{=}\sqrt{N}\bar{f}_{\boldsymbol{\sigma}}\left(0\right)+\bar{f}_{\boldsymbol{\sigma}}^{\prime}\left(0\right)+g_{3}\left(\boldsymbol{\sigma}\right)-g_{4}\left(\boldsymbol{\sigma}\right)\label{eq:11}\\
 & \overset{\left(\mathcal{A}_{3}\right),\left(\mathcal{A}_{4}\right)}{=}H_{N}\left(\sqrt{N}\boldsymbol{\sigma}\right)+\frac{1}{\sqrt{N}}H_{N}^{\prime}\left(\sqrt{N}\boldsymbol{\sigma}\right)-C_{0}+o\left(1\right).\nonumber 
\end{align}

Now, as $N\to\infty$, 
\begin{align*}
\sqrt{N}\bar{f}_{\boldsymbol{\sigma}}^{+}\left(Y_{\boldsymbol{\sigma}}\right) & \overset{\eqref{eq:5}}{=}\sqrt{N}\bar{f}_{\boldsymbol{\sigma},apx}^{+}\left(Y_{\boldsymbol{\sigma}}\right)+\sqrt{N}\bar{f}_{\boldsymbol{\sigma}}^{\left(2\right)}\left(Y_{\boldsymbol{\sigma}}\right)+\bar{f}_{\boldsymbol{\sigma}}^{\left(1\right)}\left(Y_{\boldsymbol{\sigma}}\right)\\
 & \overset{\left(\mathcal{A}_{2}\right)}{\leq}\sqrt{N}\bar{f}_{\boldsymbol{\sigma},apx}^{+}\left(Y_{\boldsymbol{\sigma}}\right)+g_{5}\left(\boldsymbol{\sigma}\right)+g_{6}\left(\boldsymbol{\sigma}\right)\\
 & \overset{\left(\mathcal{A}_{5}\right),\left(\mathcal{A}_{6}\right)}{=}\sqrt{N}\bar{f}_{\boldsymbol{\sigma},apx}^{+}\left(Y_{\boldsymbol{\sigma}}\right)+O\left(N^{1-3\alpha}\right),
\end{align*}
and
\begin{align*}
\inf_{x:\,\left\Vert x\right\Vert =N^{-\alpha}}\sqrt{N}\bar{f}_{\boldsymbol{\sigma}}^{+}\left(x\right) & \overset{\eqref{eq:5}}{=}\inf_{x:\,\left\Vert x\right\Vert =N^{-\alpha}}\left\{ \sqrt{N}\bar{f}_{\boldsymbol{\sigma},quad}\left(x\right)+\bar{f}_{\boldsymbol{\sigma},lin}^{\prime}\left(x\right)+\sqrt{N}\bar{f}_{\boldsymbol{\sigma}}^{\left(2\right)}\left(x\right)+\bar{f}_{\boldsymbol{\sigma}}^{\left(1\right)}\left(x\right)\right\} \\
 & \geq\sqrt{N}\bar{f}_{\boldsymbol{\sigma}}\left(0\right)+\bar{f}_{\boldsymbol{\sigma}}^{\prime}\left(0\right)+g_{7}\left(\boldsymbol{\sigma}\right)+g_{8}\left(\boldsymbol{\sigma}\right)-g_{5}\left(\boldsymbol{\sigma}\right)-g_{6}\left(\boldsymbol{\sigma}\right)\\
 & \overset{\left(\mathcal{A}_{i}\right),i=5,...,8}{\geq}H_{N}\left(\sqrt{N}\boldsymbol{\sigma}\right)+\frac{1}{\sqrt{N}}H_{N}^{\prime}\left(\sqrt{N}\boldsymbol{\sigma}\right)+\frac{1}{2}K_{p,\delta}N^{1-2\alpha}\left(1+o\left(1\right)\right).
\end{align*}
Therefore, 
\[
\sqrt{N}\bar{f}_{\boldsymbol{\sigma}}^{+}\left(Y_{\boldsymbol{\sigma}}\right)\leq\inf_{x:\,\left\Vert x\right\Vert =N^{-\alpha}}\sqrt{N}\bar{f}_{\boldsymbol{\sigma}}^{+}\left(x\right),
\]
and the minimum of $\sqrt{N}\bar{f}_{\boldsymbol{\sigma}}^{+}$ in
$\mathcal{B}_{N-1}\left(0,N^{-\alpha}\right)$ is attained at an interior
point $x_{*}$. 

Define $\mathscr{G}_{N,L}\left(\sqrt{N}\boldsymbol{\sigma}\right)$
to be the point corresponding to $x_{*}$ on the sphere (i.e., $\theta_{\boldsymbol{\sigma}}\circ S\circ P_{E}\left(x_{*}\right)$,
see (\ref{eq:33})). For large enough $N$, the fact that $\left\Vert x_{*}\right\Vert <N^{-\alpha}$
implies that point $(2)$ of (\ref{eq:28}) is satisfied for the good
critical point $\sqrt{N}\boldsymbol{\sigma}$. We also have that 
\[
\sup_{x\in B_{E}^{N-1}\left(0,N^{-\alpha}\right)}\left|\sqrt{N}\bar{f}_{\boldsymbol{\sigma}}^{+}\left(x\right)-\sqrt{N}\bar{f}_{\boldsymbol{\sigma},apx}^{+}\left(x\right)\right|\overset{\eqref{eq:5}}{\leq}g_{5}\left(\boldsymbol{\sigma}\right)+g_{6}\left(\boldsymbol{\sigma}\right)\leq o\left(1\right),\,\,\,\mbox{as }N\to\infty.
\]
Combined with (\ref{eq:11}), this implies that point $(1)$ of (\ref{eq:28})
is satisfied. This
completes the proof of Lemma \ref{lem:good_suff}. \qed

\subsection{\label{sub:Auxiliary-results}Auxiliary results for Lemma \ref{lem:good_whp}}

In order to prove Lemma \ref{lem:good_whp}, we will show that the
expected number of critical points $\boldsymbol{\sigma}\in\mathscr{C}_{N}\left(L\right)$
which are not good goes to $0$ as $N\to\infty$. The following lemma,
in the basis of our proof, is obtained by an application of a variant
of the Kac-Rice formula \cite[Theorem 12.1.1]{RFG}. We define $f(\boldsymbol{\sigma})=f_{N}(\boldsymbol{\sigma})$
as the unit variance random field on $\mathbb{S}^{N-1}\triangleq\left\{ \boldsymbol{\sigma}\in\mathbb{R}^{N}:\,\left\Vert \boldsymbol{\sigma}\right\Vert _{2}=1\right\} $
given by $f\left(\boldsymbol{\sigma}\right)=\frac{1}{\sqrt{N}}H_{N}\left(\sqrt{N}\boldsymbol{\sigma}\right)$.
Recall that $\omega_{N}$, given in (\ref{eq:omega_vol}), is the
surface area of the $N-1$-dimensional unit sphere.
\begin{lem}
\label{lem:K-R}For $2\leq i\leq8$ and $L>0$, 
\begin{align}
 & \mathbb{E}\left\{ \#\left\{ \sqrt{N}\boldsymbol{\sigma}\in\mathscr{C}_{N}\left(L\right):\,g_{i}\left(\boldsymbol{\sigma}\right)\notin B_{i},\,g_{1}\left(\boldsymbol{\sigma}\right)\in B_{1}\right\} \right\} \leq\omega_{N}\left(\left(N-1\right)p\left(p-1\right)\right)^{\frac{N-1}{2}}\varphi_{\nabla f(\mathbf n)}\left(0\right)\nonumber \\
 & \quad\times\mathbb{E}\left\{ \left|\det\left(\frac{\nabla^{2}f(\mathbf n)}{\sqrt{\left(N-1\right)p\left(p-1\right)}}\right)\right|\mathbf{1}\Big\{\left|\sqrt{N}f(\mathbf n)-m_{N}\right|<L,\,g_{i}\left(\mathbf{n}\right)\notin B_{i},\,g_{1}\left(\mathbf{n}\right)\in B_{1}\Big\}\,\Bigg|\,\nabla f(\mathbf n)=0\right\} ,\label{eq:35}
\end{align}
where $\varphi_{\nabla f(\mathbf n)}\left(0\right)$ is
the Gaussian density of $\nabla f(\mathbf n)$ at $0$,
$\mathbf{n}=\left(0,...,0,1\right)$,  $(E_i(\boldsymbol{\sigma}))_{i=1}^{N-1}$ is an arbitrary piecewise smooth orthonormal frame field on the sphere (w.r.t the standard Riemannian metric) and
\begin{equation}
\label{eq:gradHess11}
\nabla f\left(\boldsymbol{\sigma}\right)=\left(E_{i}f\left(\boldsymbol{\sigma}\right)\right)_{i=1}^{N-1},\,\,\nabla^{2}f\left(\boldsymbol{\sigma}\right)=\left(E_{i}E_{j}f\left(\boldsymbol{\sigma}\right)\right)_{i,j=1}^{N-1}.
\end{equation}

 In
addition, for $i=1$  (\ref{eq:35}) still holds 
if we remove the condition $g_{1}\left(\boldsymbol{\sigma}\right)\in B_{1}$
from both sides of the equation. Lastly, (\ref{eq:35}) holds as an equality if we remove the
indicator from the right-hand side and the requirements on $g_{i}\left(\boldsymbol{\sigma}\right)$
and $g_{1}\left(\boldsymbol{\sigma}\right)$ from left-hand side (in which case it expresses the expectation of the number of points in $\mathscr{C}_{N}\left(L\right)$).
\end{lem}
We note that the last case with equality in Lemma \ref{lem:K-R} also follows from \cite[Eq. (3.21)]{A-BA-C} by summing over $k$. The proof of Lemma \ref{lem:K-R} is postponed to Appendix II. We remark
that, in fact, it seems that (\ref{eq:35}) also holds as an equality.
However, proving this requires a tedious inspection of certain non-degeneracy
conditions related to \cite[Theorem 12.1.1]{RFG} which we prefer
avoiding since the easier to prove inequality above suffices to us.
We will prove in Section \ref{sub:lemmas} the following three
lemmas, using Lemma \ref{lem:K-R} for the first two.
\begin{lem}
\label{lem:g1B1}For any $L>0$,
\begin{equation}
\lim_{N\to\infty}\mathbb{E}\left\{ \#\left\{ \sqrt{N}\boldsymbol{\sigma}\in\mathscr{C}_{N}\left(L\right):\,g_{1}\left(\boldsymbol{\sigma}\right)\notin B_{1}\right\} \right\} =0.\label{eq:45}
\end{equation}

\end{lem}
 
\begin{lem}
\label{lem:g4B4}For $i=4,6$ and any $L>0$,
\[
\lim_{N\to\infty}\mathbb{E}\left\{ \#\left\{ \sqrt{N}\boldsymbol{\sigma}\in\mathscr{C}_{N}\left(L\right):\,g_{i}\left(\boldsymbol{\sigma}\right)\notin B_{i},\,g_{1}\left(\boldsymbol{\sigma}\right)\in B_{1}\right\} \right\} =0.
\]

\end{lem}
 
\begin{lem}
\label{lem:giBi}For $i\neq1,\,4,\,6$ we have the following. For
any $L>0$, there exists a sequence $e_{\epsilon,\delta}\left(N\right)$
satisfying $\lim_{N\to\infty}e_{\epsilon,\delta}\left(N\right)=0$,
such that 
\begin{equation}
\mathbb{P}\left\{ \left.g_{i}\left(\mathbf{n}\right)\notin B_{i}\,\right|\,\nabla^{2}f(\mathbf n)=\mathbf{A}_{N-1},\,\nabla f(\mathbf n)=0,\,f(\mathbf n)=u\right\} \leq e_{\epsilon,\delta}\left(N\right),\label{eq:6-1}
\end{equation}
uniformly over all $N\geq1$, $u\in\frac{1}{\sqrt{N}}\left(m_{N}-L,\,m_{N}+L\right)$,
and $\mathbf{A}_{N-1}\in B_{1}$.
\end{lem}

\subsection{\label{sub:pf_good_whp}Proof of Lemma \ref{lem:good_whp} assuming
	Lemmas \ref{lem:g1B1}, \ref{lem:g4B4} and \ref{lem:giBi}}

By Lemmas \ref{lem:g1B1} and \ref{lem:g4B4} the mean number of critical
points $\sqrt{N}\boldsymbol{\sigma}\in\mathscr{C}_{N}\left(L\right)$
for which $\mathcal{A}_{i}$ does not occur for some $i\in\left\{ 1,4,6\right\} $
goes to $0$ as $N\to\infty$. We shall prove that the mean number
of critical points $\sqrt{N}\boldsymbol{\sigma}\in\mathscr{C}_{N}\left(L\right)$
such that $g_{i}\left(\boldsymbol{\sigma}\right)\notin B_{i}$ and
$g_{1}\left(\boldsymbol{\sigma}\right)\in B_{1}$ also goes to $0$
as $N\to\infty$, for any $i\in\left\{ 1,...,8\right\} \setminus\left\{ 1,4,6\right\} $.
From this it will follow that the mean number of critical points $\sqrt{N}\boldsymbol{\sigma}\in\mathscr{C}_{N}\left(L\right)$
which are not good goes to $0$ as $N\to\infty$, and by Markov's
inequality Lemma \ref{lem:good_whp} will follow.

Using Lemma \ref{lem:K-R} and  Lemma \ref{lem:giBi}, by conditioning
in (\ref{eq:35}) on $f\left(\mathbf{n}\right)$ and $\nabla^{2}f\left(\mathbf{n}\right)$,
in addition to $\nabla f\left(\mathbf{n}\right)=0$, we have that
\begin{align*}
	& \!\!\!\!\!\! \mathbb{E}\left\{ \#\left\{ \sqrt{N}\boldsymbol{\sigma}\in\mathscr{C}_{N}\left(L\right):\,g_{i}\left(\boldsymbol{\sigma}\right)\notin B_{i},\,g_{1}\left(\boldsymbol{\sigma}\right)\in B_{1}\right\} \right\} \\
	& \leq e_{\epsilon,\delta}\left(N\right)\cdot\omega_{N}\left(\left(N-1\right)p\left(p-1\right)\right)^{\frac{N-1}{2}}\varphi_{\nabla f\left(\mathbf{n}\right)}\left(0\right)\\
	&\mathbb{E}\left\{ \left|\det\left(\frac{\nabla^{2}f\left(\mathbf{n}\right)}{\sqrt{\left(N-1\right)p\left(p-1\right)}}\right)\right|\mathbf{1}\Big\{\left|\sqrt{N}f\left(\mathbf{n}\right)-m_{N}\right|<L\Big\}\,\Bigg|\,\nabla f\left(\mathbf{n}\right)=0\right\} \\
	& =e_{\epsilon,\delta}\left(N\right)\cdot\mathbb{E}\left\{ \#\left\{ \boldsymbol{\sigma}\in\mathscr{C}_{N}\left(L\right)\right\} \right\} ,
\end{align*}
for any $i\in\left\{ 1,...,8\right\} \setminus\left\{ 1,4,6\right\} $.
By Proposition \ref{prop:intensity}, $\mathbb{E}\left\{ \#\left\{ \boldsymbol{\sigma}\in\mathscr{C}_{N}\left(L\right)\right\} \right\} $
converges to a finite number, and thus the expectation above goes
to $0$ as $N\to\infty$. This completes the proof. \qed

\subsection{\label{subsec:covs}Covariances and the conditional law of the Hessian}

In this section we state two results about the covariance of the (normalized) Hamiltonian, its gradient and its Hessian at a point and about the conditional law of the Hessian given the value of the Hamiltonian at a point. These result will be used in the proof of the lemmas of Section \ref{sub:Auxiliary-results}. 

With the map
\begin{align*}
	P:\,&\{x\in\mathbb R^{N-1}: \Vert x \Vert_2 < 1\}\longrightarrow \left\{ \boldsymbol{\sigma}\in\mathbb{S}^{N-1}:\, \langle \boldsymbol{\sigma},\mathbf n \rangle>0\right\},\\
	&(x_1,...,x_{N-1})\longmapsto \left(x_{1},...,x_{N-1},\sqrt{1-\left\Vert x\right\Vert ^{2}}\right),
\end{align*}
we have that
\[
f\circ P(x)=\bar{f}_{\mathbf{n}}\left(x\right),
\]
where  $\bar{f}_{\boldsymbol{\sigma}}(x)$ is defined in \eqref{eq:33}.

We note that there exists a smooth  orthonormal frame field $E=(E_i)$ defined on some neighborhood of $\mathbf n$ in $\mathbb S^{N-1}$ such that\footnote{This can be seen by the following. Letting $\left\{ \frac{\partial}{\partial x_{i}}\right\} _{i=1}^{N-1}$ denote the pushforward of $\left\{ \frac{d}{d x_{i}}\right\} _{i=1}^{N-1}$ by $P$ we have that at the north pole, $\left\{  \frac{\partial}{\partial x_{i}} (\mathbf n)\right\} _{i=1}^{N-1}$ is an orthonormal frame. For any point $\boldsymbol{\sigma}$ in a small neighborhood of $\mathbf n$ we can define an orthonormal frame as the parallel transport of $\left\{ \frac{\partial}{\partial x_{i}}(\mathbf n)\right\} _{i=1}^{N-1}$ along the geodesic connecting $\mathbf n$ and $\boldsymbol{\sigma}$. This yields a smooth orthonormal frame field on this neighborhood, say $E_i (\boldsymbol{\sigma})=\sum_{j=1}^{N-1}a_{ij} (\boldsymbol{\sigma})\frac{\partial}{\partial x_{j}}(\boldsymbol{\sigma})$, $i=1,...,N-1$. Working with the coordinate system $P$ one can verify that at $x=0$ the Christoffel symbols $\Gamma_{ij}^k$ are equal to $0$, and therefore (see e.g. \cite[Eq. (2), P. 53]{DoCarmo}) the derivatives $\frac{d}{dx_k}a_{ij}(P(x))$ at $x=0$ are also equal to $0$.}
\begin{equation}
\label{eq:E}\left\{ f\left(\mathbf{n}\right),\nabla f\left(\mathbf{n}\right),\nabla^{2}f\left(\mathbf{n}\right)\right\}   =\left\{ \bar{f}_{\mathbf{n}}\left(0\right),\nabla\bar{f}_{\mathbf{n}}\left(0\right),\nabla^{2}\bar{f}_{\mathbf{n}}\left(0\right)\right\} ,
\end{equation}
where  $\nabla\bar{f}_{\mathbf{n}}(0)$ and $\nabla^{2}\bar{f}_{\mathbf{n}}(0)$
are the usual (Euclidean) gradient and Hessian, while $\nabla f\left(\mathbf{n}\right)$ and $\nabla^2 f\left(\mathbf{n}\right)$ are defined in \eqref{eq:gradHess11}.

%\begin{comment}
%	
%	The projection map
%	\[
%	P'(x_1,...,x_N)=(x_1,...,x_{N-1})
%	\]
%	forms a chart on the upper hemisphere $V=\left\{ %\boldsymbol{\sigma}\in\mathbb{R}^{N}:\,\left\Vert \boldsymbol{\sigma}\right\Vert %_{2}=1,\, \langle \boldsymbol{\sigma},\mathbf n \rangle>0\right\}$.
%	Let $\left\{ \frac{\partial}{\partial x_{i}}(\boldsymbol{\sigma})\right\} %_{i=1}^{N-1}$
%	be the natural basis on $V$ corresponding to this chart, and
%	for any $\boldsymbol{\sigma}\in V$ set $\mathfrak %G(\boldsymbol{\sigma})=\{\mathfrak g_{ij}(\boldsymbol{\sigma})\}_{i,j=1}^{N-1}$
%	with
%	\begin{equation}
%	\mathfrak g_{ij}(\boldsymbol{\sigma})=\mathfrak %g\left(\frac{\partial}{\partial %x_{i}}(\boldsymbol{\sigma}),\frac{\partial}{\partial %x_{j}}(\boldsymbol{\sigma})\right)=\begin{cases}
%	\frac{\sigma_{i}\sigma_{j}}{1-\sum_{j=1}^{N-1}\sigma_{j}^{2}} & ,\,\,i\neq j\\
%	1+\frac{\sigma_{i}^{2}}{1-\sum_{j=1}^{N-1}\sigma_{j}^{2}} & ,\,\,i=j,
%	\end{cases}\label{eq:q1}
%	\end{equation}
%	where $\mathfrak g$ is the standard Riemannian metric on the sphere. Defining
%	$\{\mathfrak a_{ij}(\boldsymbol{\sigma})\}_{i,j=1}^{N-1}=\mathfrak %A(\boldsymbol{\sigma})=\left(\mathfrak G(\boldsymbol{\sigma})\right)^{-1/2}$
%	(where we take the principal square root), we have that
%	\[
%	E_{i}(\boldsymbol{\sigma})=\sum_{j=1}^{N-1}\mathfrak %a_{ij}(\boldsymbol{\sigma})\frac{\partial}{\partial x_{j}}(\boldsymbol{\sigma})
%	\]
%	is an orthonormal frame field (on $V$). At $\boldsymbol{\sigma}=\mathbf{n}$
%	we have that $E_{i}(\mathbf{n})=\frac{\partial}{\partial x_{i}}(\mathbf{n})$.
%\end{comment}
Thus, if we set in Lemma \ref{lem:cov_fbar} below $x=0$ in the first three equations, we have that Lemma \ref{lem:cov_fbar} and Corollary \ref{cor:1} follow from \cite[Lemma 3.2]{A-BA-C} (where a different normalization is used for GOE matrices). The proof of \cite[Lemma 3.2]{A-BA-C} consists of computing  derivatives and using the well-known relation (cf. \cite[eq. (5.5.4)]{RFG}),
\begin{align}
\label{eq:derivatives}	
& \mathbb E\left\{ \frac{d^{k}}{d x_{i_{1}}\cdots d x_{i_{k}}}\bar{f}_{\mathbf n}\left(x\right)\cdot\frac{d^{l}}{d y_{i_{1}}\cdots d y_{i_{l}}}\bar{f}_{\mathbf n}\left(y\right)\right\}\\ &\quad =\frac{d^{k}}{d x_{i_{1}}\cdots d x_{i_{k}}}\frac{d^{l}}{d y_{i_{1}}\cdots d y_{i_{l}}}\mathbb E \left\{ \bar{f}_{\mathbf n}\left(x\right)\cdot\bar{f}_{\mathbf n}\left(y\right)\right\}.\nonumber
\end{align}
The first equation in Lemma \ref{lem:cov_fbar} with general $x$ follows from the definition of $W(x,y)$ given in \eqref{eq:W}. From \eqref{eq:derivatives}, the second and third equations in Lemma  \ref{lem:cov_fbar} follow with general $x$.
\begin{lem}
	\cite[Lemma 3.2]{A-BA-C}\label{lem:cov_fbar} For any $x\in \mathcal B_{N-1}\left(0,1\right)$,
	\begin{align*}
		\mathbb{E}\left\{ \bar{f}_{\mathbf n}\left(x\right)\cdot\bar{f}_{\mathbf n}\left(0\right)\right\}  & =\left(W\left(0,x\right)\right)^{p},\\
		\mathbb{E}\left\{ \bar{f}_{\mathbf n}\left(x\right)\cdot\frac{d}{d x_{i}}\bar{f}_{\mathbf n}\left(0\right)\right\}  & =p\left(W\left(0,x\right)\right)^{p-1}x_{i},\\
		\mathbb{E}\left\{ \bar{f}_{\mathbf n}\left(x\right)\cdot\frac{d}{d x_{i}}\frac{d}{d x_{j}}\bar{f}_{\mathbf n}\left(0\right)\right\}  & =-\delta_{ij}p\left(W\left(0,x\right)\right)^{p}+p\left(p-1\right)\left(W\left(0,x\right)\right)^{p-2}x_{i}x_{j},\\
		\mathbb{E}\left\{ \frac{d}{d x_{i}}\bar{f}_{\mathbf n}\left(0\right)\cdot\frac{d}{d x_{j}}\bar{f}_{\mathbf n}\left(0\right)\right\}  & =\delta_{ij}p,\\
		\mathbb{E}\left\{ \frac{d}{d x_{i}}\frac{d}{d x_{j}}\bar{f}_{\mathbf n}\left(0\right)\cdot\frac{d}{d x_{k}}\bar{f}_{\mathbf n}\left(0\right)\right\}  & =0.\\
		\mathbb{E}\left\{ \frac{d}{d x_{i}}\frac{d}{d x_{j}}\bar{f}_{\mathbf n}\left(0\right)\cdot\frac{d}{d x_{k}}\frac{d}{d x_{l}}\bar{f}_{\mathbf n}\left(0\right)\right\}  & =p\left(p-1\right)\left[\delta_{ij}\delta_{kl}+\delta_{il}\delta_{jk}+\delta_{ik}\delta_{jl}\right]+p\delta_{ij}\delta_{kl}.
	\end{align*}
\end{lem}
\begin{cor}
	\cite[Lemma 3.2]{A-BA-C}\label{cor:1} The gradient $\nabla f(\mathbf n)$
	is a centered Gaussian vector with i.i.d entries of variance $p$,
	and it is independent of $\nabla^{2}f(\mathbf n)$ and $f(\mathbf n)$.
	Conditional on $f(\mathbf n)=u$, $\nabla^{2}f(\mathbf n)$
	has the same distribution as 
	\[
	\sqrt{\left(N-1\right)p\left(p-1\right)}\mathbf{M}-up\mathbf{I},
	\]
	where $\mathbf{M}=\mathbf{M}_{N-1}$ is a GOE matrix and $\mathbf{I}=\mathbf{I}_{N-1}$
	is the identity matrix, both of dimension $N-1$. 
\end{cor}

\subsection{\label{sub:Metric-entropies}Metric entropies}

The last ingredient we need before turning to the proof of Lemmas \ref{lem:g1B1}, \ref{lem:g4B4} and \ref{lem:giBi} are bounds on metric entropies. Those will be used to prove the bounds related to $g_{5}\left(\boldsymbol{\sigma}\right)$ and $g_{6}\left(\boldsymbol{\sigma}\right)$.

For a random field $w:\,T\to\mathbb{R}$, the canonical (pseudo) metric
is defined by 
\begin{equation}
d_{w}\left(x,y\right)\triangleq\sqrt{\mathbb{E}\left\{ \left(w\left(x\right)-w\left(y\right)\right)^{2}\right\} },\,\,\,\forall x,\,y\in T.\label{eq:can_met}
\end{equation}

\begin{lem}
	\label{lem:generic entropy bd}Suppose $w$ is a Gaussian field on
	the Euclidean ball of radius $R$,
	\[
	T\triangleq\mathcal{B}_{N}\left(0,R\right)=\left\{ x\in\mathbb{R}^{N}:\,\left\Vert x\right\Vert \leq R\right\} ,
	\]
	such that 
	\begin{equation}
	d_{w}\left(x,y\right)\leq\tau\cdot\left\Vert x-y\right\Vert ,\,\,\,\forall x,\,y\in T,\label{eq:1}
	\end{equation}
	for some $\tau>0$. Then
	\[
	\mathbb{E}\left\{ \sup_{x\in T}w\left(x\right)\right\} \leq \kappa \tau R\sqrt{N},
	\]
	where $\kappa>0$ is a universal constant.
\end{lem}

\begin{proof}
	
For any $\epsilon>0$, the metric entropy (or $\epsilon$-covering
number) of $T$, denoted by $N\left(T,d_{w},\epsilon\right)$, is
the smallest number of balls,
\[
\mathcal{B}_{d_{w}}\left(x,\epsilon\right)\triangleq\left\{ y\in T:\,d_{w}\left(x,y\right)\leq\epsilon\right\} ,
\]
required to cover $T$. The log-entropy is defined by
\[
H\left(\epsilon\right)=H\left(T,d_{w},\epsilon\right)\triangleq\log\left(N\left(T,d_{w},\epsilon\right)\right).
\]

With $d_{E}\left(x,y\right)$ denoting the Euclidean distance, from
(\ref{eq:1}) we have 
\[
N\left(T,d_{w},\epsilon\right)\leq N\left(T,d_{E},\tau^{-1}\epsilon\right).
\]
Let $P\left(T,d_{E},\epsilon\right)$ be the packing number of $T$;
that is, the largest number of disjoint $\epsilon$-balls (w.r.t $d_{E}$)
in $T$. Then we have that
\[
N\left(T,d_{E},\tau^{-1}\epsilon\right)\leq P\left(T,d_{E},\tau^{-1}\epsilon/2\right).
\]
From volume considerations,
\[
P\left(T,d_{E},\epsilon/2\right)\leq\frac{{\rm Vol}\left(\mathcal{B}_{N}\left(0,R\right)\right)}{{\rm Vol}\left(\mathcal{B}_{N}\left(0,\tau^{-1}\epsilon/2\right)\right)}\leq\left(\frac{R}{\tau^{-1}\epsilon/2}\right)^{N}.
\]
Thus,
\[
H\left(T,d_{w},\epsilon\right)\leq N\log\left(\frac{2\tau R}{\epsilon}\right).
\]

Dudley's entropy bound (see \cite[Theorem 1.3.3]{RFG}) then gives
\[
\mathbb{E}\left\{ \sup_{x\in T}w\left(x\right)\right\} \leq \kappa'\int_{0}^{\tau R}\sqrt{H\left(T,d_{w},\epsilon\right)}d\epsilon\leq2\tau RK'\sqrt{N}\int_{0}^{1/2}\sqrt{\log\left(\frac{1}{\omega}\right)}d\omega\triangleq \kappa\tau R\sqrt{N},
\]
where $\kappa'>0$ is a universal constant and where we used the change
of variables $\omega=\epsilon/\left(2\tau R\right)$.
\end{proof}

We complement Lemma \ref{lem:generic entropy bd} by the following one.
\begin{lem}
	\label{lem:entropy_est}Let $R\in\left(0,1\right)$. Then for any
	$x,\,y\in\mathcal{B}_{N-1}\left(0,R\right)$, 
	\begin{align}
		d_{\bar{f}_{\boldsymbol{\sigma}}^{\left(1\right)}}\left(x,y\right) & \leq a_{p}^{\left(1\right)}R\left\Vert x-y\right\Vert ,\label{eq:d1bd}\\
		d_{\bar{f}_{\boldsymbol{\sigma}}^{\left(2\right)}}\left(x,y\right) & \leq a_{p}^{\left(2\right)}R^{2}\left\Vert x-y\right\Vert ,\label{eq:d2bd}
	\end{align}
	for appropriate constants $a_{p}^{\left(1\right)},\,a_{p}^{\left(2\right)}>0$.
\end{lem}

\begin{proof}

For $k=1$, $2$, denote 
\[
\varphi_{\left(k\right)}\left(x,y\right)=d_{\bar{f}_{\boldsymbol{\sigma}}^{\left(k\right)}}\left(x,y\right).
\]
Suppose that the one-sided derivative 
\begin{equation}
\left.\frac{d}{dt}\right|_{t=0^{+}}\varphi_{\left(k\right)}\left(x,x+tv\right)\triangleq\lim_{t\to0^{+}}\frac{\varphi_{\left(k\right)}\left(x,x+tv\right)-\varphi_{\left(k\right)}\left(x,x\right)}{t},\label{eq:72}
\end{equation}
exists and its absolute value is bounded by $m>0$ uniformly in $v\in\mathbb{R}^{N-1}$
with $\left\Vert v\right\Vert =1$ and $x\in\mathcal{B}_{N-1}\left(0,R\right)$.
Then, for any $x_{1},\,x_{2}\in\mathcal{B}_{N-1}\left(0,R\right)$,
setting $v_{12}=\frac{x_{2}-x_{1}}{\left\Vert x_{2}-x_{1}\right\Vert }$
and $x\left(s\right)=x_{1}+sv_{12}$, we have that
\[
d_{\bar{f}_{\boldsymbol{\sigma}}^{\left(k\right)}}\left(x_{1},x_{2}\right)\leq\int_{0}^{\left\Vert x_{2}-x_{1}\right\Vert }\left|\left.\frac{d}{dt}\right|_{t=0^{+}}\varphi_{\left(k\right)}\left(x\left(s\right),x\left(s\right)+tv\right)\right|ds\leq m\left\Vert x_{2}-x_{1}\right\Vert .
\]
Hence, the lemma follows if we prove the existence of (\ref{eq:72})
and an appropriate upper bound on its absolute value (for $k=1,2$).

From the expressions that we derive below for $\varphi_{\left(k\right)}^{2}\left(x,x+tv\right)$
(see (\ref{eq:75}) and (\ref{eq:77})) it follows that the latter possesses
derivatives up to second order. We will also show that 
\begin{equation}
\varphi_{\left(k\right)}^{2}\left(x,x\right)=\left.\frac{d}{dt}\right|_{t=0}\varphi_{\left(k\right)}^{2}\left(x,x+tv\right)=0\,\,\,\,\,\mbox{and}\,\,\,\,\,\left.\frac{d^{2}}{dt^{2}}\right|_{t=0}\varphi_{\left(k\right)}^{2}\left(x,x+tv\right)\geq0.\label{eq:73}
\end{equation}
From this, by a Taylor expansion of $\varphi_{\left(k\right)}^{2}\left(x,x+tv\right)$,
we have that 
\begin{align}
	\left.\frac{d}{dt}\right|_{t=0^{+}}\varphi_{\left(k\right)}\left(x,x+tv\right) & =\lim_{t\to0^{+}}\frac{\sqrt{\varphi_{\left(k\right)}^{2}\left(x,x+tv\right)}-0}{t}\nonumber \\
	& =\lim_{t\to0^{+}}\frac{\sqrt{\frac{1}{2}t^{2}\cdot\left.\frac{d^{2}}{ds^{2}}\right|_{s=0}\varphi_{\left(k\right)}^{2}\left(x,x+sv\right)+o\left(t^{2}\right)}}{t}\nonumber \\
	& =\sqrt{\frac{1}{2}\left.\frac{d^{2}}{ds^{2}}\right|_{s=0}\varphi_{\left(k\right)}^{2}\left(x,x+sv\right)}.\label{eq:74}
\end{align}
Therefore, the lemma follows if we show, in addition to (\ref{eq:73}),
that 
\begin{equation}
\varphi_{\left(k\right)}^{2}\left(x,x+tv\right)\leq C_{k}R^{2k}t^{2}+o\left(t^{2}\right),\mbox{\,\,\,\ as }t\to0,\label{eq:78}
\end{equation}
for any $x\in\mathcal{B}_{N-1}\left(0,R\right)$ with and appropriate
constant $C_{k}>0$ independent of $R$, $x$ and $N$. Note that
if the lemma holds for $R<1/2$, then it is true in general, where
if needed the constants $a_{p}^{\left(k\right)}$ are increased. Therefore,
it is in fact enough to prove the uniform bound (\ref{eq:78}) assuming
$R<1/2$, and this is what we shall do.

Using Lemma \ref{lem:cov_fbar} we obtain, for $v\in\mathbb{R}^{N-1}$
with $\left\Vert v\right\Vert =1$,
\begin{align}
	\varphi_{\left(1\right)}^{2}\left(x,x+tv\right) & =\mathbb{E}\left\{ \left(\bar{f}_{\boldsymbol{\sigma}}^{\left(1\right)}\left(x\right)\right)^{2}\right\} +\mathbb{E}\left\{ \left(\bar{f}_{\boldsymbol{\sigma}}^{\left(1\right)}\left(x+tv\right)\right)^{2}\right\} -2\mathbb{E}\left\{ \bar{f}_{\boldsymbol{\sigma}}^{\left(1\right)}\left(x\right)\bar{f}_{\boldsymbol{\sigma}}^{\left(1\right)}\left(x+tv\right)\right\} \label{eq:76}\\
	& =2-2\left(W\left(x,x+tv\right)\right)^{p}+pt^{2}\nonumber \\
	& +2p\left(1-\left\Vert x\right\Vert ^{2}\right)^{\frac{p-1}{2}}\left\langle x,tv\right\rangle \nonumber \\
	& -2p\left(1-\left\Vert x+tv\right\Vert ^{2}\right)^{\frac{p-1}{2}}\left\langle x+tv,tv\right\rangle ,\nonumber 
\end{align}
where $W\left(x,y\right)$ is defined in (\ref{eq:W}). 

Note that 
\[
W\left(x,x+tv\right)=1-\frac{1}{2}\left(\frac{\left\langle x,v\right\rangle ^{2}}{1-\left\Vert x\right\Vert ^{2}}+\left\Vert v\right\Vert ^{2}\right)t^{2}+o\left(t^{2}\right).
\]
Also, any of the summands in (\ref{eq:76}) can be written as a quadratic
function in $t$ plus a remainder term of $o\left(t^{2}\right)$.
Doing so and combining like terms we arrive at 
\begin{align}
	\varphi_{\left(1\right)}^{2}\left(x,x+tv\right) & =p\left(\frac{\left\langle x,v\right\rangle ^{2}}{1-\left\Vert x\right\Vert ^{2}}+2-2\left(1-\left\Vert x\right\Vert ^{2}\right)^{\frac{p-1}{2}}+2\left(p-1\right)\left(1-\left\Vert x\right\Vert ^{2}\right)^{\frac{p-3}{2}}\left\langle x,v\right\rangle ^{2}\right)t^{2}+o\left(t^{2}\right)\nonumber \\
	& \leq p\left(\frac{\left\Vert x\right\Vert ^{2}}{1-\left\Vert x\right\Vert ^{2}}+2-2\left(1-\left\Vert x\right\Vert ^{2}\right)^{\frac{p-1}{2}}+2\left(p-1\right)\left(1-\left\Vert x\right\Vert ^{2}\right)^{\frac{p-3}{2}}\left\Vert x\right\Vert ^{2}\right)t^{2}+o\left(t^{2}\right).\label{eq:75}
\end{align}
Under the assumption that $R<1/2$, 
\[
\varphi_{\left(1\right)}^{2}\left(x,x+tv\right)\leq C_{1}R^{2}t^{2}+o\left(t^{2}\right),
\]
for an appropriate constant $C_{1}>0$ independent of $x\in\mathcal{B}_{N-1}\left(0,R\right)$
and $N$. Combined with (\ref{eq:74}), this proves that (\ref{eq:73})
and (\ref{eq:d1bd}) hold.

The proof for the case $k=2$ is similar. Lemma \ref{lem:cov_fbar}
and some algebra give, for $v\in\mathbb{R}^{N-1}$ with $\left\Vert v\right\Vert =1$,
\begin{align*}
	& \negthickspace\negthickspace\negthickspace\negthickspace\varphi_{\left(2\right)}^{2}\left(x,x+tv\right)-\varphi_{\left(1\right)}^{2}\left(x,x+tv\right)\\
	& =\frac{1}{2}p\left(p-1\right)\left(-2\left\langle x,v\right\rangle ^{2}t^{2}+\left(2t\left\langle x,v\right\rangle +t^{2}\right)^{2}+2\left\Vert x\right\Vert ^{2}t^{2}\right)+\frac{1}{4}p^{2}\left(2t\left\langle x,v\right\rangle +t^{2}\right)^{2}\\
	& -p\left(\left(1-\left\Vert x\right\Vert ^{2}\right)^{p/2}-\left(1-\left\Vert x\right\Vert ^{2}-2t\left\langle x,v\right\rangle -t^{2}\right)^{p/2}\right)\left(2t\left\langle x,v\right\rangle +t^{2}\right)\\
	& +p\left(p-1\right)\left(1-\left\Vert x\right\Vert ^{2}\right)^{p/2-1}\left(2t\left\Vert x\right\Vert ^{2}\left\langle x,v\right\rangle +\left\langle x,v\right\rangle ^{2}t^{2}\right)\\
	& -p\left(p-1\right)\left(1-\left\Vert x\right\Vert ^{2}-2t\left\langle x,v\right\rangle -t^{2}\right)^{p/2-1}\left(2\left\Vert x\right\Vert ^{2}+3t\left\langle x,v\right\rangle +t^{2}\right)\left(t\left\langle x,v\right\rangle +t^{2}\right).
\end{align*}
Expanding in $t$ and combining like terms yields
\begin{align}
	& \negthickspace\negthickspace\negthickspace\negthickspace\varphi_{\left(2\right)}^{2}\left(x,x+tv\right)-\varphi_{\left(1\right)}^{2}\left(x,x+tv\right)=\left[p\left(-2p+1\right)\left\langle x,v\right\rangle ^{2}+p\left(p-1\right)\left\Vert x\right\Vert ^{2}\right.\nonumber \\
	& \left.+p\left(4p-3\right)\left(p-2\right)\left\langle x,v\right\rangle ^{2}\left\Vert x\right\Vert ^{2}+p\left(p-1\right)\left(p-2\right)\left\Vert x\right\Vert ^{4}+o\left(\left\Vert x\right\Vert ^{4}\right)\right]t^{2}+o\left(t^{2}\right).\label{eq:77}
\end{align}
Combined with (\ref{eq:75}) this gives, for $R<1/2$,
\[
\varphi_{\left(2\right)}^{2}\left(x,x+tv\right)\leq C_{2}R^{4}t^{2}+o\left(t^{2}\right),
\]
with an appropriate $C_{2}>0$, for all $x\in\mathcal{B}_{N-1}\left(0,R\right)$.
The proof is thus completed.
\end{proof}

\subsection{\label{sub:lemmas}Proof of Lemmas \ref{lem:g1B1}, \ref{lem:g4B4},
and \ref{lem:giBi}}

In this section we prove the three lemmas in the title. The proofs
will rely on Corollary \ref{cor:1}, which describes the covariance
structure of the (normalized) Hamiltonian, its gradient and Hessian at a point,
and on the metric entropy bounds of Lemma \ref{lem:entropy_est} and the related Lemma \ref{lem:generic entropy bd}. Everywhere in this section we will assume that the choice of the orthonormal frame field $(E_i(\boldsymbol{\sigma}))_{i=1}^{N-1}$ is such that \eqref{eq:E} holds. Note that any of the cases $i=1,...,8$,
corresponding to $g_{i}$ and $B_{i}$, is covered by exactly one
of the lemmas. We therefore give the proofs according to these cases,
while stating in the titles of the proofs the corresponding lemma.

\subsection*{Proof of the case $i=1$ (Lemma \ref{lem:g1B1})}

 Recall that 
 \[
 g_{1}\left(\boldsymbol{\sigma}\right)=\nabla^{2}\bar{f}_{\boldsymbol{\sigma}}\left(0\right),\qquad B_{1}=\sqrt{N}\mathcal{V}\left(\delta\right),
 \]
  where $\mathcal{V}\left(\delta\right)$ is the
  set of real, symmetric matrices with eigenvalues in the interval
  \[
  \left(pE_{0}-\left(2\sqrt{p\left(p-1\right)}+\delta\right),\,pE_{0}+2\sqrt{p\left(p-1\right)}+\delta\right).
  \]
Using Lemma \ref{lem:K-R} and the fact that by Corollary \ref{cor:1}, $f(\mathbf n)$ and $\nabla f(\mathbf n)$ are independent, we have that, denoting $D_{N}=\frac{1}{\sqrt{N}}\left(m_{N}-L,\,m_{N}+L\right)$,
\begin{align}
 & \mathbb{E}\left\{ \#\left\{ \sqrt{N}\boldsymbol{\sigma}\in\mathscr{C}_{N}\left(L\right):\,g_{1}\left(\boldsymbol{\sigma}\right)\notin B_{1}\right\} \right\} \leq\omega_{N}\left(\left(N-1\right)p\left(p-1\right)\right)^{\frac{N-1}{2}}\varphi_{\nabla f\left(\mathbf{n}\right)}\left(0\right)\nonumber \\
 & \quad\times\int_{D_{N}}du\frac{1}{\sqrt{2\pi}}e^{-\frac{u^{2}}{2}}\mathbb{E}\left\{ \left|\det\left(\frac{\nabla^{2}f\left(\mathbf{n}\right)}{\sqrt{\left(N-1\right)p\left(p-1\right)}}\right)\right|\mathbf{1}\Big\{ g_{1}\left(\mathbf{n}\right)\notin B_{1}\Big\}\,\Bigg|\,\nabla f\left(\mathbf{n}\right)=0,\,f\left(\mathbf{n}\right)=u\right\} .\label{eq:54}
\end{align}
From Corollary \ref{cor:1}, under the conditioning, the normalized
Hessian matrix in (\ref{eq:54}) is a GOE matrix plus $-\frac{\gamma_{p}u}{\sqrt{N-1}}\mathbf{I}$ (where we recall that $\gamma_{p}=\sqrt{p/\left(p-1\right)}$).
Since $E_{0}>E_{\infty}=2/\gamma_{p}$, for a small enough constant
$a>0$, for any $u\in D_{N}$, $-\frac{\gamma_{p}u}{\sqrt{N-1}}>2+a$.
Thus, Lemma \ref{lem:19} and \eqref{eq:E}, by which $g_1(\mathbf n)=\nabla^{2}f(\mathbf{n})$, imply that the conditional expectation
in (\ref{eq:54}) is bounded by 
\[
e^{-cN}\mathbb{E}\left\{ \left|\det\left(\frac{\nabla^{2}f\left(\mathbf{n}\right)}{\sqrt{\left(N-1\right)p\left(p-1\right)}}\right)\right|\,\Bigg|\,\nabla f\left(\mathbf{n}\right)=0,\,f\left(\mathbf{n}\right)=u\right\} ,
\]
for an appropriate constant $c>0$, for all $N>N_{0}\left(L\right)$
(independent of $u$), and uniformly in $u\in D_{N}$. Hence, the left-hand side  of \eqref{eq:54} is less than $e^{-cN}$ times its right-hand side with the indicator removed. The latter is equal to $\mathbb{E}\left\{ \#\left\{ \boldsymbol{\sigma}\in\mathscr{C}_{N}\left(L\right)\right\} \right\}$, as follows from the case with equality in Lemma \ref{lem:K-R}.
Since $\mathbb{E}\left\{ \#\left\{ \boldsymbol{\sigma}\in\mathscr{C}_{N}\left(L\right)\right\} \right\} $
converges, by Proposition \ref{prop:intensity}, to a finite number
as $N\to\infty$, the proof is completed.\qed

\subsection*{Proof of the case $i=2$ (Lemma \ref{lem:giBi})}

Recall that
\[
g_{2}\left(\boldsymbol{\sigma}\right)=Y_{\boldsymbol{\sigma}}:=-\frac{1}{\sqrt{N}}\left(\nabla^{2}\bar{f}_{\boldsymbol{\sigma}}\left(0\right)\right)^{-1}\nabla\bar{f}_{\boldsymbol{\sigma}}^{\prime}\left(0\right),\qquad B_{2}=\mathcal{B}_{N-1}\left(0,N^{-\alpha}\right).
\]

The minimal eigenvalue of any $\mathbf{A}_{N-1}\in B_1=\sqrt{N}\mathcal{V}\left(\delta\right)$ is larger than $\sqrt NK_{p,\delta}>0$ (see \eqref{eq:Kp} and recall that $\delta$ was chosen in the beginning of Section \ref{sec:pfProp5} such that $K_{p,\delta}>0$). Thus, for any such $\mathbf{A}_{N-1}$ and $u\in\mathbb R$, conditional
on 
\begin{equation}
\nabla^{2}f\left(\mathbf{n}\right)=\mathbf{A}_{N-1},\,\nabla f\left(\mathbf{n}\right)=0,\,f\left(\mathbf{n}\right)=u,\label{eq:41}
\end{equation}
using \eqref{eq:E} we have the following stochastic domination
\[
\left\Vert Y_{\mathbf{n}}\right\Vert \leq\frac{1}{NK_{p,\delta}}\left\Vert \nabla\bar{f}_{\mathbf{n}}^{\prime}\left(0\right)\right\Vert .
\]

Since, by Corollary \ref{cor:1} and \eqref{eq:E},  $\nabla\bar{f}_{\mathbf{n}}^{\prime}\left(0\right)\sim N\left(0,pI\right)$,
\begin{align*}
&\mathbb{P}\left\{ \left.g_{2}\left(\mathbf{n}\right)\notin B_{2}\,\right|\,\nabla^{2}f(\mathbf n)=\mathbf{A}_{N-1},\,\nabla f(\mathbf n)=0,\,f(\mathbf n)=u\right\} \\
&\quad \leq \mathbb{P}\left\{ Q_{N-1}\geq N^{1-\alpha}\frac{K_{p,\delta}}{\sqrt{p}}\right\} ,
\end{align*}
where $Q_{N-1}$ has standard Chi distribution with $N-1$ degrees
of freedom. This completes the proof since we assumed (see the beginning of Section \ref{sec:pfProp5}) that $\alpha\in (1/3,1/2)$ and since $Q_{N-1}/\sqrt N$ converges in probability to $1$. \qed

\subsection*{Proof of the case $i=3$ (Lemma \ref{lem:giBi})}

Recall that 
\begin{align*}
	g_{3}\left(\boldsymbol{\sigma}\right)  =\Delta_{\boldsymbol{\sigma}}+\frac{p}{2\sqrt{N}}{\rm Tr}\left(\left(\nabla^{2}\bar{f}_{\boldsymbol{\sigma}}\left(0\right)\right)^{-1}\right),  \quad B_{3}=\left(-N^{-\frac{1}{2}+\epsilon},N^{-\frac{1}{2}+\epsilon}\right),
\end{align*}
where
\[
	\Delta_{\boldsymbol{\sigma}}  =\frac{\sqrt{N}}{2}Y_{\boldsymbol{\sigma}}^{T}\nabla^{2}\bar{f}_{\boldsymbol{\sigma}}\left(0\right)Y_{\boldsymbol{\sigma}}+\left\langle \nabla\bar{f}_{\boldsymbol{\sigma}}^{\prime}\left(0\right),Y_{\boldsymbol{\sigma}}\right\rangle.
\]
By Corollary \ref{cor:1} and \eqref{eq:E}, $\nabla\bar{f}_{\mathbf{n}}^{\prime}\left(0\right)\sim N\left(0,pI\right)$.
By definition $\nabla\bar{f}_{\mathbf{n}}^{\prime}\left(0\right)$
and $f\left(\mathbf{n}\right)$ are independent. Therefore,
conditional on $\eqref{eq:41}$, $\nabla\bar{f}_{\mathbf{n}}^{\prime}\left(0\right)$
has the same law as without the conditioning. Another way to write this law, and this is what we will use, is that conditional on $\eqref{eq:41}$, $\nabla\bar{f}_{\mathbf{n}}^{\prime}\left(0\right)$ has the same distribution as 
\[
\sqrt{p}\sum_{i=1}^{N-1}W_{i}a_{i},
\]
with $W_{i}\sim N\left(0,1\right)$ i.i.d and $a_{1},...,a_{N-1}$ being
an orthonormal basis composed of the eigenvectors of $\mathbf{A}_{N-1}$.
Thus, under the conditioning and for any $u\in \mathbb R$ and $\mathbf{A}_{N-1}\in B_1$, denoting the eigenvalues of $\mathbf{A}_{N-1}$ by $\lambda_{i}\left(\mathbf{A}_{N-1}\right)$,
$g_{3}\left(\mathbf{n}\right)$ has the same law as
\[
\frac{p}{2\sqrt{N}}\sum_{i=1}^{N-1}\left(1-W_{i}^{2}\right)/\lambda_{i}\left(\mathbf{A}_{N-1}\right).
\]
In particular, since the minimal eigenvalue of any $\mathbf{A}_{N-1}\in B_1=\sqrt{N}\mathcal{V}\left(\delta\right)$ is larger than $\sqrt NK_{p,\delta}$ (see \eqref{eq:Kp}), under the conditioning, 
$g_{3}\left(\mathbf{n}\right)$ is a centered variable with second moment bounded by 
\[
\left(\frac{p}{2NK_{p,\delta}}\right)^{2}\sum_{i=1}^{N-1}\mathbb{E}\left\{ \left(1-W_{i}^{2}\right)^{2}\right\} =\left(\frac{p}{2NK_{p,\delta}}\right)^{2}\cdot2\left(N-1\right).
\]
The lemma follows from
this.\qed

\subsection*{Proof of the case $i=4$ (Lemma \ref{lem:g4B4})}

Recall that
\begin{align*}
	g_{4}\left(\boldsymbol{\sigma}\right)  =\frac{p}{2\sqrt{N}}{\rm Tr}\left(\left(\nabla^{2}\bar{f}_{\boldsymbol{\sigma}}\left(0\right)\right)^{-1}\right), \quad B_{4}=\left(C_{0}-\tau_{\epsilon,\delta}\left(N\right),C_{0}+\tau_{\epsilon,\delta}\left(N\right)\right),
\end{align*}
where $C_{0}$ is defined in (\ref{eq:C0}) and $\tau_{\epsilon,\delta}\left(N\right)$
is a sequence of numbers such that $\lim_{N\to\infty}\tau_{\epsilon,\delta}\left(N\right)=0$ which is assumed to be large enough whenever needed.

By Lemma \ref{lem:K-R} we have, denoting $D_{N}=\frac{1}{\sqrt{N}}\left(m_{N}-L,\,m_{N}+L\right)$,
\begin{align}
 & \mathbb{E}\left\{ \#\left\{ \sqrt{N}\boldsymbol{\sigma}\in\mathscr{C}_{N}\left(L\right):\,g_{4}\left(\boldsymbol{\sigma}\right)\notin B_{4},g_{1}\left(\boldsymbol{\sigma}\right)\in B_{1}\right\} \right\} \leq\omega_{N}\left(\left(N-1\right)p\left(p-1\right)\right)^{\frac{N-1}{2}}\varphi_{\nabla f\left(\mathbf{n}\right)}\left(0\right)\nonumber \\
 & \quad\times\int_{D_{N}}\frac{1}{\sqrt{2\pi}}e^{-\frac{u^{2}}{2}}\mathbb{E}\left\{ \left|\det\left(\frac{\nabla^{2}f\left(\mathbf{n}\right)}{\sqrt{\left(N-1\right)p\left(p-1\right)}}\right)\right|\mathbf{1}\Big\{ g_{4}\left(\mathbf{n}\right)\notin B_{4},g_{1}\left(\mathbf{n}\right)\in B_{1}\Big\}\,\Bigg|\,f\left(\mathbf{n}\right)=u\right\} du,\nonumber%\label{eq:54-1}
\end{align}
where we used the fact that $\nabla f\left(\mathbf{n}\right)$ is
independent of the expression in
the expectation (as follows from Corollary \ref{cor:1} and \eqref{eq:E}). Note that, for large enough $K>0$, the term preceding the integral above is bounded by $e^{KN}$ and so do $e^{-u^2/2}$ and the determinant above, for $u\in D_N$ and $\nabla^{2}f\left(\mathbf{n}\right)=g_1(\mathbf n)\in B_1$ (see \eqref{eq:E}).

Hence, in order to finish the proof it will be enough to show that, uniformly in $u\in D_N$,

\begin{equation}
\lim_{N\to\infty}\frac{1}{N}\log\left(\mathbb{P}\left\{ g_{4}\left(\mathbf{n}\right)\notin B_{4},g_{1}\left(\mathbf{n}\right)\in B_{1}\,\big|\,f\left(\mathbf{n}\right)=u\right\} \right)=-\infty.\label{eq:56}
\end{equation}

For any general real, symmetric matrix $\mathbf{A}$ let $\lambda_{i}\left(\mathbf{A}\right)$
denote the eigenvalues of $\mathbf{A}$. Let $\mathbf{M}$ be an $N-1$
dimensional GOE matrix and set $\tilde{\mathbf{M}}_{u}\triangleq\sqrt{\left(N-1\right)p\left(p-1\right)}\mathbf{M}-pu\mathbf{I}$.
From Corollary \ref{cor:1} and \eqref{eq:E}, as $N\to\infty$, 
\begin{align}
 & \mathbb{P}\left\{ \left.g_{4}\left(\mathbf{n}\right)\notin B_{4}\,,g_{1}\left(\mathbf{n}\right)\in B_{1}\,\right|\,f\left(\mathbf{n}\right)=u\right\} \nonumber \\
 & =\mathbb{P}\left\{ \left|\frac{p}{2\sqrt{N}}\sum_{i=1}^{N-1}\frac{1}{\lambda_{i}\left(\tilde{\mathbf{M}}_{u}\right)}-C_{0}\right|>\tau_{\epsilon,\delta}\left(N\right),\,\tilde{\mathbf{M}}_{u}\in\sqrt{N}\mathcal{V}\left(\delta\right)\right\} \nonumber \\
 & =\mathbb{P}\left\{ \left|\left(1+o\left(1\right)\right)\frac{1}{2}\gamma_{p}\cdot\frac{1}{N}\sum_{i=1}^{N-1}\frac{1}{\lambda_{i}\left(\mathbf{M}\right)+\gamma_{p}E_{0}\left(1+o\left(1\right)\right)}-C_{0}\right|>\tau_{\epsilon,\delta}\left(N\right),\,\mathbf{M}\in\mathcal{V}'\left(\delta\right)\right\} ,\label{eq:a2}
\end{align}
where $\gamma_{p}=\sqrt{p/(p-1)}$ and we define $\mathcal{V}'\left(\delta\right)=\mathcal{V}'\left(\delta,N\right)$
as the set of real, symmetric matrices with all eigenvalues in 
\[
\left(-\left(2+\frac{\delta}{\sqrt{p\left(p-1\right)}}+o\left(1\right)\right),2+\frac{\delta}{\sqrt{p\left(p-1\right)}}+o\left(1\right)\right),
\]
and all the $o\left(1\right)$ terms above are uniform in
$u\in D_{N}$.

The restriction $\delta\in\left(0,p\left(E_{0}-E_{\infty}\right)\right)$ (which was made in the beginning of Section \ref{sec:pfProp5}),
is such that there exists a constant $c_{\delta}>0$
such that, for large enough $N$, the event $\mathbf{M}\in\mathcal{V}'\left(\delta\right)$
is contained in the event $F_{N}=\left\{ \forall i:\,\lambda_{i}\left(\mathbf{M}+\gamma_{p}E_{0}\mathbf{I}\right)\geq c_{\delta}\right\} $.
Thus, if we are able to show that
\begin{equation}
\lim_{N\to\infty}\frac{1}{N}\log\left(\mathbb{P}\left\{ \left|\frac{1}{2}\gamma_{p}\cdot\frac{1}{N}\sum_{i=1}^{N-1}\frac{1}{\lambda_{i}\left(\mathbf{M}\right)+\gamma_{p}E_{0}}-C_{0}\right|>\tau_{\epsilon,\delta}\left(N\right),\,F_{N}\right\} \right)=-\infty,\label{eq:9}
\end{equation}
with some sequence $\tau_{\epsilon,\delta}\left(N\right)$ which converges
to $0$ as $N\to\infty$, the lemma follows (where to account for
the $o\left(1\right)$ terms we maybe need to increase $\tau_{\epsilon,\delta}\left(N\right)$
compared to (\ref{eq:a2})). Define $\bar{h}\left(x\right)=1/\left(x+\gamma_{p}E_{0}\right)$. Let $\rho\in(0,c_{\delta})$ and assume  $h\left(x\right)$ is some bounded, Lipschitz
continuous function such that for any $x\geq c_{\delta}-\gamma_{p}E_{0}-\rho$, it holds that $h\left(x\right)=\bar{h}\left(x\right)$. Then on the event $F_{N}$,
\[
\frac{1}{N}\sum_{i=1}^{N-1}\frac{1}{\lambda_{i}\left(\mathbf{M}\right)+\gamma_{p}E_{0}}=\left\langle L_{\mathbf{M}},h\right\rangle ,
\]
where $L_{\mathbf{M}}=\frac{1}{N}\sum_{i=1}^{N-1}\delta_{\lambda_{i}\left(\mathbf{M}\right)}$
is the empirical measure of eigenvalues of $\mathbf{M}$.

Hence, the probability in (\ref{eq:9}) is at most
\[
\mathbb{P}\left\{ \left|\frac{1}{2}\gamma_{p}\left\langle L_{\mathbf{M}},h\right\rangle -C_{0}\right|\geq\tau_{\epsilon,\delta}\left(N\right)\right\} .
\]
Wigner's law \cite{wigner} yields  
\[
\lim_{N\to\infty}\mathbb{E}\left\{ \frac{1}{2}\gamma_{p}\left\langle L_{\mathbf{M}},h\right\rangle \right\} =C_{0},
\]
where we recall that $C_{0} =\frac{1}{2}\gamma_{p}\left\langle \mu^*,h\right\rangle$   with $\mu^*$ denoting the semicircle law. Therefore, from standard concentration inequalities as in \cite[Theorem 2.3.5]{Matrices},
(\ref{eq:9}) follows, and the proof is completed, if we choose
$\tau_{\epsilon,\delta}\left(N\right)\to0$ such that 
\[
\lim_{N\to\infty}N\left(\tau_{\epsilon,\delta}\left(N\right)-\left|\mathbb{E}\left\{ \frac{1}{2}\gamma_{p}\left\langle L_{\mathbf{M}},h\right\rangle \right\} -C_{0}\right|\right)^{2}=\infty.
\]
\qed

\subsection*{Proof of the case $i=5$ (Lemma \ref{lem:giBi})}

Recall that
\begin{align*}
	g_{5}\left(\boldsymbol{\sigma}\right)  =\sup_{x\in\mathcal{B}_{N-1}\left(0,N^{-\alpha}\right)}\left|\bar{f}_{\boldsymbol{\sigma}}^{\left(1\right)}\left(x\right)\right|,  \quad B_{5}=\left(-CN^{\frac{1}{2}-2\alpha},CN^{\frac{1}{2}-2\alpha}\right),
\end{align*}
and
\begin{equation*}
\bar{f}_{\boldsymbol{\sigma}}^{\left(1\right)}\left(x\right)=\bar{f}_{\boldsymbol{\sigma}}^{\prime}\left(x\right)-\bar{f}_{\boldsymbol{\sigma},lin}^{\prime}\left(x\right)=\bar{f}_{\boldsymbol{\sigma}}^{\prime}\left(x\right)-\bar{f}_{\boldsymbol{\sigma}}^{\prime}\left(0\right)-\left\langle \nabla\bar{f}_{\boldsymbol{\sigma}}^{\prime}\left(0\right),x\right\rangle,
\end{equation*}
where $C>0$ is a constant which can be assumed to be large enough
whenever needed.

The random variable $g_{5}\left(\mathbf{n}\right)$ is measurable with respect to the process $\{\bar{f}_{\mathbf{n}}^{\prime}\left(x\right)\}_x$ and is therefore
independent of all the variables in the conditioning of (\ref{eq:6-1}).
From Lemmas \ref{lem:generic entropy bd} and \ref{lem:entropy_est}
with $R=N^{-\alpha}$, $\tau=a_p^{(1)}R$, and $\kappa$ being the universal constant of Lemma \ref{lem:generic entropy bd},
\[
\mathbb{E}\left\{ \sup_{x\in\mathcal{B}_{N-1}\left(0,N^{-\alpha}\right)}\bar{f}_{\mathbf{n}}^{\left(1\right)}\left(x\right)\right\} \leq \kappa \tau RN^{\frac{1}{2}}
= \kappa a_{p}^{\left(1\right)}N^{\frac{1}{2}-2\alpha},
\]
and 
\[
\sup_{x\in\mathcal{B}_{N-1}\left(0,N^{-\alpha}\right)}\mathbb{E}\left\{ \left(\bar{f}_{\mathbf{n}}^{\left(1\right)}\left(x\right)\right)^{2}\right\} =\sup_{x\in\mathcal{B}_{N-1}\left(0,N^{-\alpha}\right)}\left(d_{\bar{f}_{\mathbf{n}}^{\left(1\right)}}\left(0,x\right)\right)^{2}\leq \left( a_{p}^{\left(1\right)}R^2 \right)^2= \left(a_{p}^{\left(1\right)} \right)^2 N^{-4\alpha}.
\]
Thus, from the Borell-TIS
inequality \cite[Theorem 2.1.1]{RFG},
\[
\lim_{N\to\infty}\mathbb{P}\left\{ \sup_{x\in\mathcal{B}_{N-1}\left(0,N^{-\alpha}\right)}\bar{f}_{\mathbf{n}}^{\left(1\right)}\left(x\right)>2\kappa a_{p}^{\left(1\right)}N^{\frac{1}{2}-2\alpha}\right\} =0.
\]
From symmetry of the field $\bar{f}_{\mathbf{n}}^{\left(1\right)}\left(x\right)$,
one can treat the infimum similarly. Then, using a union bound the
lemma follows.\qed

\subsection*{Proof of the case $i=6$ (Lemma \ref{lem:g4B4})}

Recall that
\begin{align*}
	g_{6}\left(\boldsymbol{\sigma}\right)  =\sup_{x\in\mathcal{B}_{N-1}\left(0,N^{-\alpha}\right)}\left|\sqrt{N}\bar{f}_{\boldsymbol{\sigma}}^{\left(2\right)}\left(x\right)\right|,  \quad B_{6}=\left(-CN^{1-3\alpha},CN^{1-3\alpha}\right),
\end{align*}
and 
\[
\bar{f}_{\boldsymbol{\sigma}}^{\left(2\right)}\left(x\right)=\bar{f}_{\boldsymbol{\sigma}}\left(x\right)-\bar{f}_{\boldsymbol{\sigma}}\left(0\right)-\left\langle \nabla\bar{f}_{\boldsymbol{\sigma}}\left(0\right),x\right\rangle -\frac{1}{2}x^{T}\nabla^{2}\bar{f}_{\boldsymbol{\sigma}}\left(0\right)x,
\]
where $C>0$ is a constant which can be assumed to be large enough
whenever needed.

By Lemma \ref{lem:K-R},
\begin{align*}
	& \mathbb{E}\left\{ \#\left\{ \sqrt{N}\boldsymbol{\sigma}\in\mathscr{C}_{N}\left(L\right):\,g_{6}\left(\boldsymbol{\sigma}\right)\notin B_{6},\,g_{1}\left(\boldsymbol{\sigma}\right)\in B_{1}\right\} \right\} \\
	& \quad \leq\omega_{N}\left(\left(N-1\right)p\left(p-1\right)\right)^{\frac{N-1}{2}}\varphi_{\nabla f(\mathbf n)}\left(0\right)\nonumber 
	\mathbb{E}\left\{ \left|\det\left(\frac{\nabla^{2}f(\mathbf n)}{\sqrt{\left(N-1\right)p\left(p-1\right)}}\right)\right|\cdots \right.\\
	&\quad \left. \mathbf{1}\Big\{\left|\sqrt{N}f(\mathbf n)-m_{N}\right|<L,\,g_{6}\left(\mathbf{n}\right)\notin B_{6},\,g_{1}\left(\mathbf{n}\right)\in B_{1}\Big\}\,\Bigg|\,\nabla f(\mathbf n)=0\right\} .
\end{align*}
By the similar argument to the one that was used in the proof for the case $i=4$ to show that \eqref{eq:56} was enough to finish the proof there, here we have that the proof will follow if we can show that for any $K>0$ choosing $C$ (with we which defined $B_6$) large enough, for large
$N$, 
\begin{equation}
\mathbb{P}\left\{ g_{6}\left(\mathbf{n}\right)\notin B_{6}\,\Bigg|\,\nabla f\left(\mathbf{n}\right)=0\right\} \leq e^{-KN}.\label{eq:l1}
\end{equation}

From
Lemmas \ref{lem:generic entropy bd} and \ref{lem:entropy_est}, with $R=N^{-\alpha}$, $\tau=a_p^{(2)}R^2$ and with $\kappa$ being the universal constant of Lemma \ref{lem:generic entropy bd},
\begin{equation*}
	\mathbb{E}\left\{ \sup_{x\in\mathcal{B}_{N-1}\left(0,N^{-\alpha}\right)}\sqrt{N}\bar{f}_{\mathbf{n}}^{\left(2\right)}\left(x\right)\right\} \leq \kappa\tau R N= \kappa a_p^{(2)} N^{1-3\alpha},
\end{equation*}
and
\begin{align*}
\sup_{x\in\mathcal{B}_{N-1}\left(0,N^{-\alpha}\right)}\mathbb{E}\left\{ \left( \sqrt{N}\bar{f}_{\mathbf{n}}^{\left(2\right)}\left(x\right) \right)^{2}\right\} &=N \sup_{x\in\mathcal{B}_{N-1}\left(0,N^{-\alpha}\right)}\left(d_{\bar{f}_{\mathbf{n}}^{\left(2\right)}\left(x\right)}\left(0,x\right)\right)^{2}\\
&\leq N\left( a_p^{(2)}R^3 \right)^2= \left(a_{p}^{\left(2\right)} \right)^2 N^{1-6\alpha}.
\end{align*}

By the Borell-TIS inequality \cite[Theorem 2.1.1]{RFG}, for large $C$,
\begin{align*}
 \mathbb{P}\left\{ \sup_{x\in\mathcal{B}_{N-1}\left(0,N^{-\alpha}\right)}\sqrt{N}\bar{f}_{\mathbf{n}}^{\left(2\right)}\left(x\right)>CN^{1-3\alpha}\right\} 
  \leq \exp\left\{ -\frac{1}{2} N \left(\frac{C-\kappa a_p^{(2)}}{a_p^{(2)}}\right)^2\right\}.
\end{align*}

By symmetry, the same holds if we replace the supremum above by an infimum. Thus, from the union bound,
\begin{align*}
	\mathbb{P}\left\{ g_6(\mathbf n)\notin B_6 \right\}  &= \mathbb{P}\left\{ \sup_{x\in\mathcal{B}_{N-1}\left(0,N^{-\alpha}\right)}\left|\sqrt{N}\bar{f}_{\mathbf{n}}^{\left(2\right)}\left(x\right)\right|>CN^{1-3\alpha}\right\} \\
	&\leq 2\exp\left\{ -\frac{1}{2} N \left(\frac{C-\kappa a_p^{(2)}}{a_p^{(2)}}\right)^2\right\}=:S(C,N).
\end{align*}

What remains is to show the same conditional on $\nabla f (\mathbf n)=0$, as in \eqref{eq:l1}. Since $\bar{f}_{\mathbf{n}}^{\left(2\right)}\left(x\right)$ is a continuous field, it is enough to show that for any finite set of points $x_1,...,x_k\in \mathcal{B}_{N-1}\left(0,N^{-\alpha}\right)$,
\[
\mathbb{P}\left\{ \left. \sup_{i\leq k}\left|\sqrt{N}\bar{f}_{\mathbf{n}}^{\left(2\right)}\left(x_i\right)\right|>CN^{1-3\alpha} \, \right| \, \nabla f\left(\mathbf{n}\right)=0 \right\}\leq S(C,N).
\]

Equivalently, it is sufficient to show that
\begin{equation}
\mathbb{P}\left\{ \left. \left(\sqrt{N}\bar{f}_{\mathbf{n}}^{\left(2\right)}\left(x_i\right)\right)_{i\leq k} \in \left[-CN^{1-3\alpha},CN^{1-3\alpha}\right]^k \, \right| \, \nabla f\left(\mathbf{n}\right)=0 \right\} \label{eq:l2}
\end{equation}
can only decrease by removing the conditioning.

We note that 
\begin{align}
\label{eq:k8} \left( \sqrt{N}\bar{f}_{\mathbf{n}}^{\left(2\right)}\left(x_i\right) \right)_{i\leq k} &= \left( \mathbb E\left\{\left. \sqrt{N}\bar{f}_{\mathbf{n}}^{\left(2\right)}\left(x_i\right)\, \right| \,\nabla f (\mathbf n) \right\} \right)_{i\leq k} \\
&+ \left( \sqrt{N}\bar{f}_{\mathbf{n}}^{\left(2\right)}\left(x_i\right) - \mathbb E\left\{\left. \sqrt{N}\bar{f}_{\mathbf{n}}^{\left(2\right)}\left(x_i\right)\, \right| \,\nabla f (\mathbf n) \right\} \right)_{i\leq k},\nonumber
\end{align}
with the two summands in the right-hand side being independent since $\left\{\bar{f}_{\mathbf{n}}^{\left(2\right)}\left(x\right) \right\}_x$ and  $\nabla f (\mathbf n)$ are jointly Gaussian. Also, denoting the Gaussian vector in the second line above by $V_k$ we have that the conditional probability \eqref{eq:l2} is equal to
\begin{equation}
\label{eq:k9}\mathbb P\left\{  V_k\in  \left[-CN^{1-3\alpha},CN^{1-3\alpha}\right]^k \right\}.
\end{equation}
Since the set $\left[-CN^{1-3\alpha},CN^{1-3\alpha}\right]^k$ is convex and symmetric about the origin, by Anderson's inequality \cite[Corollary 3]{Anderson} and \eqref{eq:k8}, we have that \eqref{eq:k9} is bounded from below by
\[
\mathbb{P}\left\{  \left(\sqrt{N}\bar{f}_{\mathbf{n}}^{\left(2\right)}\left(x_i\right)\right)_{i\leq k} \in \left[-CN^{1-3\alpha},CN^{1-3\alpha}\right]^k  \right\} 
\]
and the proof is completed.\qed

\subsection*{Proof of the case $i=7$ (Lemma \ref{lem:giBi})}

Recall that
\[
	g_{7}\left(\boldsymbol{\sigma}\right)  =\inf_{x:\,\left\Vert x\right\Vert =N^{-\alpha}}\sqrt{N}\bar{f}_{\boldsymbol{\sigma},quad}\left(x\right)-\sqrt{N}\bar{f}_{\boldsymbol{\sigma}}\left(0\right),  \quad B_{7}=\left(\frac{1}{2}K_{p,\delta}N^{1-2\alpha},\,\infty\right),
\]
where
\[
\bar{f}_{\boldsymbol{\sigma},quad}\left(x\right)  =\bar{f}_{\boldsymbol{\sigma}}\left(0\right)+\left\langle \nabla\bar{f}_{\boldsymbol{\sigma}}\left(0\right),x\right\rangle +\frac{1}{2}x^{T}\nabla^{2}\bar{f}_{\boldsymbol{\sigma}}\left(0\right)x,
\]
and
\[
K_{p,\delta}=p\left(E_{0}-E_{\infty}\right)-\delta>0.
\]

Conditional on 
\begin{equation}
\nabla^{2}f\left(\mathbf{n}\right)=\mathbf{A}_{N-1},\,\nabla f\left(\mathbf{n}\right)=0,\,f\left(\mathbf{n}\right)=u,\label{eq:l3}
\end{equation}
by \eqref{eq:E} we have
\[
\sqrt{N}\bar{f}_{\mathbf{n},quad}\left(x\right)-\sqrt{N}\bar{f}_{\mathbf{n}}\left(0\right)=\frac{1}{2}\sqrt{N}x^{T}\mathbf{A}_{N-1}x.
\]
Assuming that $\mathbf{A}_{N-1}\in B_{1}$, the minimal eigenvalue
of $\mathbf{A}_{N-1}$ is $\sqrt{N}K_{p,\delta}$ at least. Thus,
for $x$ with $\left\Vert x\right\Vert =N^{-\alpha}$, deterministically,
\[
\frac{1}{2}\sqrt{N}x^{T}\mathbf{A}_{N-1}x\geq\frac{1}{2}K_{p,\delta}N^{1-2\alpha}.
\]
This, of course, completes the proof. \qed

\subsection*{Proof of the case $i=8$ (Lemma \ref{lem:giBi})}

Recall that
\[
g_{8}\left(\boldsymbol{\sigma}\right)  =\inf_{x:\,\left\Vert x\right\Vert =N^{-\alpha}}\bar{f}_{\boldsymbol{\sigma},lin}^{\prime}\left(x\right)-\bar{f}_{\boldsymbol{\sigma}}^{\prime}\left(0\right),  \quad B_{8}=\left(-N^{-\frac{1}{2}-\alpha+\epsilon},\,\infty\right).
\]

Note that
\[
\bar{f}_{\mathbf{n},lin}^{\prime}\left(x\right)-\bar{f}_{\mathbf{n}}^{\prime}\left(0\right)=\left\langle \nabla\bar{f}_{\mathbf{n}}^{\prime}\left(0\right),x\right\rangle 
\]
is independent of all the variables in (\ref{eq:l3}). Thus, from
Corollary \ref{cor:1} and \eqref{eq:E}, in this case the claim in the lemma is equivalent
to the statement that 
\[
\lim_{N\to\infty}\mathbb{P}\left\{ \sqrt{p}Q_{N-1}N^{-\alpha}\leq -N^{\frac{1}{2}-\alpha+\epsilon}\right\} =0,
\]
where $Q_{N-1}$ is a standard Chi variable with $N-1$ degrees of
freedom. Since $Q_{N-1}/\sqrt N\to 1$ in probability, this completes the proof.\qed

\section{Appendix I: proof of Lemma \ref{lem:apx.Ligg}}
First we remark that by conditioning on $\eta$, for any $x\in\mathbb R$,
\[
\mathbb{E}\left\{ \eta_{\infty}^{+}\left((-\infty,x)\right)\right\} \leq \mathbb E \Big\{\sum_i v(x-\eta_i)\Big\} = \int_{\mathbb{R}}e^{ay} v(x-y) dy,
\]
which is finite due to our assumption on $v(x)$; therefore
$\eta_{\infty}^{+}$ is locally finite. 

Let $g:\,\mathbb{R}\to\mathbb{R}$ be an arbitrary compactly supported,
non-negative function which will be fixed throughout the proof. Let
$\kappa>0$ be a large enough constant such that the support of $g$
is contained in $\left[-\kappa,\kappa\right]$. Denote the event 
\begin{equation*}
B=B_{L,N,\kappa}:=\left\{ \left.\eta_{N}^{+}\right|_{\left[-\kappa,\kappa\right]}=\left.\bar{\eta}_{N,L}\right|_{\left[-\kappa,\kappa\right]}\right\}.
\end{equation*}
Since this is the same event as in (\ref{eq:15}), defining $\epsilon(L)$ by  
\begin{equation}
\liminf_{N\to\infty}\mathbb{P}\left\{ B_{L,N,\kappa}\right\} =1-\epsilon\left(L\right),\label{eq:16}
\end{equation}
we have that $\epsilon\left(L\right)\to0$.
Denote $\left\langle g,\zeta\right\rangle \triangleq\int gd\zeta$
and let $\mathcal{L}_{\zeta}\left[g\right]\triangleq\mathbb{E}\left\{ \exp\left\{ -\left\langle g,\zeta\right\rangle \right\} \right\} $
be the Laplace functional of $\zeta$. Then 
\begin{equation}
\limsup_{N\to\infty}\left|\mathcal{L}_{\eta_{N}^{+}}\left[g\right]-\mathbb{E}\left\{ \mathbf{1}_{B}\exp\left\{ -\left\langle g,\eta_{N}^{+}\right\rangle \right\} \right\} \right|=\limsup_{N\to\infty}\mathbb{E}\left\{ \mathbf{1}_{B^{c}}\exp\left\{ -\left\langle g,\eta_{N}^{+}\right\rangle \right\} \right\} \le\epsilon\left(L\right).\label{eq:22}
\end{equation}

Fix some $L>0$, let $\delta>0$, and let $m_{0}:=m_{0}\left(L,\delta\right)$
be a natural number such that $\mathbb{P}\left\{ D_{L,N,\delta}\right\} >1-\delta$,
for all $N$, with 
\[
D:=D_{L,N,\delta}\triangleq\left\{ Q_{N,L}\leq m_{0}\right\} ,
\]
where $Q_{N,L}=\eta_{N}\left(\left[-L,L\right]\right)$  as defined in point (1) of the lemma. The existence
of $m_{0}$ is guaranteed by tightness  of the convergent sequence
$\eta_{N}$. We have that 
\begin{equation}
\left|\mathbb{E}\left\{ \mathbf{1}_{B}\exp\left\{ -\left\langle g,\eta_{N}^{+}\right\rangle \right\} \right\} -\mathbb{E}\left\{ \mathbf{1}_{B\cap D}\exp\left\{ -\left\langle g,\eta_{N}^{+}\right\rangle \right\} \right\} \right|\le\delta.\label{eq:21}
\end{equation}
Recall that we assumed in point (1) that $\left.\eta_{N}\right|_{\left[-L,L\right]}=\sum_{i=1}^{Q_{N,L}}\delta_{\eta_{N,L,i}}$.
With $\hat{\eta}_{N,L}\triangleq\sum_{i=1}^{Q_{N,L}}\delta_{\eta_{N,L,i}+X_{i}}$,
\[
\mathbf{1}_{B\cap D}\left|\left\langle g,\eta_{N}^{+}\right\rangle -\left\langle g,\,\hat{\eta}_{N,L}\right\rangle \right|
\]
is stochastically dominated by 
\begin{equation}
\mathbf{1}_{B\cap D}\sum_{i=1}^{m_{0}}\omega_{g}\left(\left|\bar{X}_{N,i}\left(L\right)-X_{i}\right|\right),\label{eq:20}
\end{equation}
where $\omega_{g}\left(t\right)=\sup_{\left|x-y\right|\leq t}\left|g\left(x\right)-g\left(y\right)\right|$
is the modulus of continuity of $g$. Since $g$ is uniformly continuous,
$\omega_{g}$ is continuous. Combining this with point (2) of the
lemma implies that (\ref{eq:20}) converges in probability to $0$.
Therefore, since $e^{-x}$ is bounded and continuous for $x\geq0$,
\[
\limsup_{N\to\infty}\left|\mathbb{E}\left\{ \mathbf{1}_{B\cap D}\exp\left\{ -\left\langle g,\eta_{N}^{+}\right\rangle \right\} \right\} -\mathbb{E}\left\{ \mathbf{1}_{B\cap D}\exp\left\{ -\left\langle g,\hat{\eta}_{N,L}\right\rangle \right\} \right\} \right|=0,
\]
and 
\[
\limsup_{N\to\infty}\left|\mathbb{E}\left\{ \mathbf{1}_{B\cap D}\exp\left\{ -\left\langle g,\eta_{N}^{+}\right\rangle \right\} \right\} -\mathcal{L}_{\hat{\eta}_{N,L}}\left[g\right]\right|<1-\liminf_{N\to\infty}\mathbb{P}\left\{ B_{L,N,\kappa}\cap D_{L,N,\delta}\right\} .
\]
Combining this with the bounds we have on $\mathbb{P}\left\{ D_{L,N,\delta}\right\} $
and $\mathbb{P}\left\{ B_{L,N,\kappa}\right\} $, (\ref{eq:21}),
and (\ref{eq:22}) yields, upon letting $\delta\to0$,
\[
\lim_{L\to\infty}\limsup_{N\to\infty}\left|\mathcal{L}_{\eta_{N}^{+}}\left[g\right]-\mathcal{L}_{\hat{\eta}_{N,L}}\left[g\right]\right|=0.
\]
Thus, in order to complete the proof it is sufficient to show that
for some sequence $L_{k}>0$ such that $\lim_{k\to\infty}L_{k}=\infty$,
\begin{equation}
\lim_{k\to\infty}\lim_{N\to\infty}\mathcal{L}_{\hat{\eta}_{N,L_{k}}}\left[g\right]=\mathcal{L}_{\eta_{\infty}^{+}}\left[g\right].\label{eq:23}
\end{equation}

By \cite[Lemma 11.1.II]{DaleyVereJones}, for all but a countable
set of values $L>0$, the interval $\left[-L,L\right]$ is a stochastic
continuity set of $\eta$. Thus we can choose $L_{k}$ as above such
that, from \cite[Theorem 4.2]{KallenbergMeas},
\[
\left.\eta_{N}\right|_{\left[-L_{k},L_{k}\right]}\stackrel[{\scriptstyle N\to\infty}]{d}{\to}\left.\eta\right|_{\left[-L_{k},L_{k}\right]}.
\]
Hence also
\[
\hat{\eta}_{N,L_{k}}\stackrel[{\scriptstyle N\to\infty}]{d}{\to}\sum_{i:\eta_{i}\in\left[-L_{k},L_{k}\right]}\delta_{\eta_{i}+X_{i}},
\]
from which (\ref{eq:23}) follows.\qed

\section{Appendix II: proof of Lemma \ref{lem:K-R}}

We recall that $f(\boldsymbol{\sigma})=f_{N}(\boldsymbol{\sigma})$
is the unit variance random field on $\mathbb{S}^{N-1}\triangleq\left\{ \boldsymbol{\sigma}\in\mathbb{R}^{N}:\,\left\Vert \boldsymbol{\sigma}\right\Vert _{2}=1\right\} $
given by $f\left(\boldsymbol{\sigma}\right)=\frac{1}{\sqrt{N}}H_{N}\left(\sqrt{N}\boldsymbol{\sigma}\right)$,  $(E_i(\boldsymbol{\sigma}))_{i=1}^{N-1}$ is an arbitrary piecewise smooth orthonormal frame field on the sphere (w.r.t the standard Riemannian metric), and 
%\begin{equation}
%\label{eq:gradHess}
\[
  \nabla f\left(\boldsymbol{\sigma}\right)=\left(E_{i}f\left(\boldsymbol{\sigma}\right)\right)_{i=1}^{N-1},\,\,\nabla^{2}f\left(\boldsymbol{\sigma}\right)=\left(E_{i}E_{j}f\left(\boldsymbol{\sigma}\right)\right)_{i,j=1}^{N-1}.
  \]
%\end{equation}

Provided that certain regularity conditions hold, the Kac-Rice formula \cite[Theorem 12.1.1]{RFG} expresses the mean
number of points $\boldsymbol{\sigma}_{0}$ on the sphere at which
 $\nabla f\left(\boldsymbol{\sigma}_{0}\right)=0$, $\sqrt N f\left(\boldsymbol{\sigma}_{0}\right)\in (m_N-L,m_N+L)$, and $g(\boldsymbol{\sigma}_0)\in B'$, for some random field $g(\boldsymbol{\sigma})$. For this one needs to apply \cite[Theorem 12.1.1]{RFG} with $f$ in the statement of \cite[Theorem 12.1.1]{RFG} being equal to `our' $\nabla f\left(\boldsymbol{\sigma}\right)$, and with $h(\boldsymbol{\sigma})=(\sqrt N f\left(\boldsymbol{\sigma}\right),g(\boldsymbol{\sigma}))$ and  $B=(m_N-L,m_N+L)\times B'$. This yields
 \begin{align}
 		& \mathbb{E}\left\{ \#\left\{ \sqrt{N}\boldsymbol{\sigma}\in\mathscr{C}_{N}\left(L\right):\,g(\boldsymbol{\sigma})\in B' \right\} \right\} \nonumber \\
 		&\quad =\mathbb{E}\left\{ \#\left\{ \boldsymbol{\sigma}\in\mathbb{S}^{N-1}:\nabla f\left(\boldsymbol{\sigma}\right)=0,\,h\left(\boldsymbol{\sigma}\right)\in B\right\} \right\} = \int _{\mathbb S^{N-1}}d\mu(\boldsymbol{\sigma})\varphi_{\nabla f\left(\boldsymbol{\sigma}\right)}\left(0\right) \label{eq:355} \\
 		& \quad \times\mathbb{E}\left\{ \left|\det\left(\nabla^{2}f\left(\boldsymbol{\sigma}\right)\right)\right|\mathbf{1}\Big\{\left|\sqrt{N}f\left(\boldsymbol{\sigma}\right)-m_{N}\right|<L,\,g\left(\boldsymbol{\sigma}\right)\in B'\Big\}\,\Bigg|\,\nabla f\left(\boldsymbol{\sigma}\right)=0\right\} ,\nonumber
\end{align}
where $\mu$ is the standard (Hausdorff) measure on the sphere and $\varphi_{\nabla f\left(\boldsymbol{\sigma}\right)}\left(0\right)$ is
the Gaussian density of $\nabla f\left(\boldsymbol{\sigma}\right)$ at $0$. 
Assuming that $(f(\boldsymbol{\sigma}),g(\boldsymbol{\sigma}))$ is a stationary field on the sphere, one can replace $\boldsymbol{\sigma}$ everywhere in the right-hand side above by $\mathbf n$, remove the integration, and multiply the right-hand side above by $\omega_{N}$ (see \eqref{eq:omega_vol}), the
surface area of the $N-1$-dimensional unit sphere.\footnote{We note that the integrand in \eqref{eq:355} is a continuous Radon-Nikodym derivative (as seen from applying the Kac-Rice formula \cite[Theorem 12.1.1]{RFG} to express the mean number of points as above in a subset of the sphere) and therefore it is independent, at each point $\boldsymbol{\sigma}$, of the choice of the orthonormal frame field $(E_i(\boldsymbol{\sigma}))_{i=1}^{N-1}$.}
This yields
 \begin{align}
 	& \mathbb{E}\left\{ \#\left\{ \sqrt{N}\boldsymbol{\sigma}\in\mathscr{C}_{N}\left(L\right):\,g(\boldsymbol{\sigma})\in B' \right\} \right\} =  \omega_N\varphi_{\nabla f(\mathbf n) }\left(0\right) \nonumber \\
 	& \quad \times\mathbb{E}\left\{ \left|\det\left(\nabla^{2}f(\mathbf n)\right)\right|\mathbf{1}\Big\{\left|\sqrt{N}f(\mathbf n)-m_{N}\right|<L,\,g\left(\boldsymbol{\sigma}\right)\in B'\Big\}\,\Bigg|\,\nabla f(\mathbf n)=0\right\} .\label{eq:3555}
 \end{align}

In Lemma \ref{lem:K-R} we stated one case where (\ref{eq:35}) holds as an equality. This case is equivalent to \eqref{eq:3555}  without the restrictions on $g(\boldsymbol{\sigma})$ (after dividing and multiplying by $\sqrt{Np(p-1)}$ and $(Np(p-1))^\frac{N-1}{2}$ inside the determinant and outside of it, respectively). Thus in order to prove this case what remains is  only to check the aforementioned regularity conditions. Those are not difficult to verify when we do not need to worry about the dependence of $g(\boldsymbol{\sigma})$ and $f(\boldsymbol{\sigma})$. In particular, the fact that the conditions hold here follows from the case we treat below where we show that they also hold with some $g(\boldsymbol{\sigma})$ without removing the restrictions on it.

In the other two cases in Lemma \ref{lem:K-R}, (\ref{eq:35}) needs to proved in its original from, with an inequality, for $2\leq i\leq 8$; and in a modified form without the condition $g_1(\boldsymbol{\sigma})\in B_1$ in both sides, for $i=1$. The proof of the second of the latter two cases is similar to that of the first and therefore we only treat the first. This case is equivalent to \eqref{eq:35} if $g(\boldsymbol{\sigma})=(g_i(\boldsymbol{\sigma}),g_1(\boldsymbol{\sigma}))$ and $B'=B_i^c\times B_1$, and the equality is replaced with an inequality. In order to apply the above to \eqref{eq:35} directly we will need to verify the regularity conditions and in this case this is not necessarily an easy task. Instead of doing so we relate \eqref{eq:35} to a similar formula which holds for some modified random fields for which checking the regularity conditions is much easier.

 With $B_{i}^{c}$  denoting the complement of $B_{i}$ in the corresponding Euclidean space
 (i.e., in $\mathbb{R}$ for $i>2$, $\mathbb{R}^{\left(N-1\right)^{2}}$
 for $i=1$, and $\mathbb{R}^{N-1}$ for $i=2$), let $(B_{i}^{c})_\epsilon$
 denote the set of points with distance at most $\epsilon$ from $B_{i}^{c}$.
 Let $Z_{i}$ be a ($\boldsymbol{\sigma}$-independent) Gaussian vector
 in the corresponding Euclidean space all of whose entries are i.i.d
 standard Gaussian variables. Setting
 \begin{align*}
 	h_{i}\left(\boldsymbol{\sigma}\right) & =\left(g_{i}\left(\boldsymbol{\sigma}\right),g_{1}\left(\boldsymbol{\sigma}\right),\sqrt{N}f\left(\boldsymbol{\sigma}\right)\right),\\
 	h_{i,\epsilon}\left(\boldsymbol{\sigma}\right) & =\left(g_{i}\left(\boldsymbol{\sigma}\right)+\epsilon Z_{i},g_{1}\left(\boldsymbol{\sigma}\right),\sqrt{N}f\left(\boldsymbol{\sigma}\right)\right),
 \end{align*}
 we have that, with $D_{N}=\left(m_{N}-L,\,m_{N}+L\right)$,
 \begin{align*}
 	& \mathbb{E}\left\{ \#\left\{ \sqrt{N}\boldsymbol{\sigma}\in\mathscr{C}_{N}\left(L\right):\,g_{i}\left(\boldsymbol{\sigma}\right)\notin B_{i},\,g_{1}\left(\boldsymbol{\sigma}\right)\in B_{1}\right\} \right\} \\
 	& =\mathbb{E}\left\{ \#\left\{ \boldsymbol{\sigma}\in\mathbb{S}^{N-1}:\nabla f\left(\boldsymbol{\sigma}\right)=0,\,h_{i}\left(\boldsymbol{\sigma}\right)\in B_{i}^{c}\times B_{1}\times D_{N}\right\} \right\} \\
 	& \leq \mathbb{E}\left\{ \#\left\{ \boldsymbol{\sigma}\in\mathbb{S}^{N-1}:\nabla f\left(\boldsymbol{\sigma}\right)=0,\,h_{i,\epsilon_{2}}\left(\boldsymbol{\sigma}\right)\in (B_{i}^{c})_{\epsilon_1}\times B_{1}\times D_{N}\right\} \right\} +\delta\left(\epsilon_{1},\epsilon_{2}\right),
 \end{align*}
for $\epsilon_1,\epsilon_2>0$ and an appropriate function $\delta(\epsilon_1,\epsilon_2)$ satisfying $\lim_{\epsilon_{2}\to0}\delta\left(\epsilon_{1},\epsilon_{2}\right)=0$.

Thus, provided that we can apply \eqref{eq:3555} in the current situation, we have that
 \begin{align*}
 	& \mathbb{E}\left\{ \#\left\{ \sqrt{N}\boldsymbol{\sigma}\in\mathscr{C}_{N}\left(L\right):\,g_{i}\left(\boldsymbol{\sigma}\right)\notin B_{i},\,g_{1}\left(\boldsymbol{\sigma}\right)\in B_{1}\right\} \right\} \\
 	& \leq \delta\left(\epsilon_{1},\epsilon_{2}\right)+ \omega_{N}\left(Np\left(p-1\right)\right)^{\frac{N-1}{2}}\varphi_{\nabla  f(\mathbf n)}\left(0\right) \mathbb{E}\left\{ \left|\det\left(\frac{\nabla^{2}f(\mathbf n)}{\sqrt{Np\left(p-1\right)}}\right)\right|\right.\cdots\\
 	&\left. \mathbf{1}\Big\{\left|\sqrt{N}f(\mathbf n)-m_{N}\right|<L,\,g_{i}\left(\boldsymbol{\sigma}\right)+\epsilon_{2}Z_{i}\in (B_{i}^{c})_{\epsilon_{1}},\,g_{1}\left(\boldsymbol{\sigma}\right)\in B_{1}\Big\}\,\Bigg|\,\nabla f(\mathbf n)=0\right\}.
 \end{align*}
Since $g_{i}\left(\boldsymbol{\sigma}\right)+\epsilon_{2}Z_{i}\stackrel{d}{\to}g_{i}\left(\boldsymbol{\sigma}\right)$,
as $\epsilon_{2}\to0$, and the indicator function of $(B_{i}^{c})_{\epsilon_{1}}$
is upper semi-continuous, the limit of the expectation in the left-hand side above,
as $\epsilon_{2}\to0$, is bounded from above by the same expectation
with $\epsilon_{2}=0$. Therefore, by first taking $\epsilon_{2}\to0$
and then taking $\epsilon_{1}\to0$ (and using the monotone convergence
theorem and the fact that $B_i^c$ is closed) the lemma follows.

All that remains is to verify the conditions required for the application
of \cite[Theorem 12.1.1]{RFG}, i.e., conditions (a)-(g) there, with 
$f$, $\nabla f$ and $h$ of \cite[Theorem 12.1.1]{RFG} being equal to `our' $\nabla f\left(\boldsymbol{\sigma}\right)$,
$\nabla^{2}f\left(\boldsymbol{\sigma}\right)$ and $h_{i,\epsilon}\left(\boldsymbol{\sigma}\right)$.
In particular, we need to account for $8$ different cases (of $i$). 

From the definition of the Hamiltonian (\ref{eq:Hamiltonian}), the
fact that $f\left(\boldsymbol{\sigma}\right)$ is Gaussian, and from
the Borell-TIS inequality \cite[Theorem 2.1.1]{RFG}, the components
of $\nabla f\left(\boldsymbol{\sigma}\right)$, $\nabla^{2}f\left(\boldsymbol{\sigma}\right)$
and $h_{i,\epsilon}\left(\boldsymbol{\sigma}\right)$ (in any of the cases above) are continuous and have 
finite variance at any $\boldsymbol{\sigma}$.\footnote{For the cases $i=2,3,4$ we have continuity and finite variance conditional
on $g_{1}\left(\boldsymbol{\sigma}\right)\in B_{1}$ (and not in general),
which is, of course, sufficient since we anyway work under this conditioning
and since by continuity if $g_{1}\left(\boldsymbol{\sigma}_{0}\right)\in B_{1}$
for a particular point $\boldsymbol{\sigma}_{0}$, then there exists
a neighborhood of $\boldsymbol{\sigma}_{0}$ on the sphere on which
$g_{1}\left(\boldsymbol{\sigma}\right)\in B_{1}$.} Combining this with Corollary \ref{cor:1} which assures the non-degeneracy
of the Gaussian variable $\left(f\left(\boldsymbol{\sigma}\right),\nabla f\left(\boldsymbol{\sigma}\right),\nabla^{2}f\left(\boldsymbol{\sigma}\right)\right)$,
up to symmetry of the Hessian, conditions (a)-(e) can be verified
for all the cases. Condition (f) follows since $\nabla^{2}f\left(\boldsymbol{\sigma}\right)$
is Gaussian and stationary. Condition (g) which involves the modulus of continuity of the random
fields, is verified for $f\left(\boldsymbol{\sigma}\right)$, $\nabla f\left(\boldsymbol{\sigma}\right)$
and $\nabla^{2}f\left(\boldsymbol{\sigma}\right)$ directly from the definition of the Hamiltonian (\ref{eq:Hamiltonian}).

All that we have left to show is that condition (g) holds for the
components of $g_{i}\left(\boldsymbol{\sigma}\right)+\epsilon Z_{i}$
for any $1\leq i\leq8$ and $\epsilon>0$. Since $Z_{i}$ does not depend
on $\boldsymbol{\sigma}$, the components of $g_{i}\left(\boldsymbol{\sigma}\right)+\epsilon Z_{i}$
and $g_{i}\left(\boldsymbol{\sigma}\right)$ have the same modulus
of continuity, and it is therefore enough to prove that condition
(g) holds for the components of $g_{i}\left(\boldsymbol{\sigma}\right)$.
For $i=1$ this is already proven, since $g_{1}\left(\boldsymbol{\sigma}\right)$
is equal to $\nabla^{2}f\left(\boldsymbol{\sigma}\right)$. For $i=2,3,4$
(working under the conditioning $g_{1}\left(\boldsymbol{\sigma}\right)\in B_{1}$),
the modulus of continuity of $g_{i}\left(\boldsymbol{\sigma}\right)$
can be related to that of the components of $\left(f\left(\boldsymbol{\sigma}\right),\nabla f\left(\boldsymbol{\sigma}\right),\nabla^{2}f\left(\boldsymbol{\sigma}\right)\right)$
 to prove condition (g).

In order to deal with the case $i=5,...,8$, we first note that the
modulus of continuity of the supremums in the definition of $g_{i}$
is bounded by that of the functions the supremum of which is taken.
The latter is bounded, up to a constant depending on $N$, by the
sum of moduli of continuity of the components of $f\left(\boldsymbol{\sigma}\right)$,
$\nabla f\left(\boldsymbol{\sigma}\right)$ and $\nabla^{2}f\left(\boldsymbol{\sigma}\right)$.
From this condition (g) follows in those cases too and the proof is
completed.\qed

\bibliographystyle{plain}
\bibliography{master}

\end{document}